\documentclass{article}

\usepackage{fullpage}

\usepackage{hyperref}

\usepackage{amsmath,amssymb,amsfonts, dsfont}
\usepackage{bm}
\usepackage{soul,color}
\usepackage{graphicx}
\usepackage{enumitem}
\usepackage{mathtools}
\usepackage{subcaption}
\usepackage{amsthm}

\usepackage{tikz}
\usetikzlibrary{automata,positioning, arrows}

\newcommand{\virgolette}[1]{``#1''}
\newcommand{\B}{\mathbb{B}}
\newcommand{\R}{\mathbb{R}}
\newcommand{\N}{\mathbb{N}}

\newcommand{\M}{\langle M \rangle}

\newcommand{\cL}{\mathcal{L}}
\newcommand{\cF}{\mathcal{F}}
\newcommand{\cC}{\mathcal{C}}
\newcommand{\cD}{\mathcal{D}}
\newcommand{\cK}{\mathcal{K}}
\newcommand{\cG}{\mathcal{G}}

\newcommand{\cV}{\mathcal{V}}
\newcommand{\cP}{\mathcal{P}}

\newcommand{\cA}{\mathcal{A}}

\DeclareMathOperator{\co}{conv}
\DeclarePairedDelimiterX{\inp}[2]{\langle}{\rangle}{#1, #2}
\DeclareMathOperator{\ssquare}{\scriptstyle\square}
\DeclareMathOperator{\inn}{int}

\newtheorem{thm}{Theorem}
\newtheorem{defn}{Definition}
\newtheorem{lem}{Lemma}
\newtheorem{exmp}{Example}
\newtheorem{prop}{Proposition}
\newtheorem{cor}{Corollary}
\newtheorem{assum}{Assumption}
\newtheorem{rem}{Remark}
\newtheorem{proper}{Properties}
\begin{document}

\title{Comparison of Path-Complete Lyapunov Functions via Template-Dependent Lifts}

\author{Virginie Debauche \and Matteo Della Rossa \and Raphaël M. Jungers}

\maketitle

\begin{abstract}

This paper investigates, in the context of discrete-time switching systems, the problem of comparison for path-complete stability certificates. We introduce and study abstract operations on path-complete graphs, called \emph{lifts}, which allow us to recover previous results in a general framework. Moreover, this approach highlights the existing relations between the analytical properties of the chosen set of candidate Lyapunov functions (the \emph{template}) and the admissibility of certain lifts. This provides a new methodology for the characterization of the order relation of path-complete Lyapunov functions criteria, when a particular template is chosen. We apply our results to specific templates, notably the sets of primal and dual copositive norms, providing new stability certificates for positive switching systems. These tools are finally illustrated with the aim of numerical examples.
\end{abstract}

\section{Introduction}
Switching systems not only provide a model for several physical/engineering phenomena \cite{Shorten2006,Donkers11}, but they also offer several challenging problems from a theoretic point of view \cite{Liberzon99,Jung09}. In this paper, we focus on \emph{discrete-time switching systems} of the form
\begin{equation}\label{Def:DiscreteSwitchingSystem}
x(k+1)~=~f_{\sigma(k)}(x(k)),
\end{equation}
\noindent where at time $k \in \N$, the state $x(k)$ lies in $\R^n$, and the switching signal $\sigma : \N \to \{1, \dots, M\}$ associates one of the \emph{dynamics} $F = \{ f_1, \dots, f_M\} \subseteq \cC^0(\R^n,\R^n)$ at each time step $k$. Their study offers several theoretical challenges among which the stability problem has especially attracted the attention of many researchers. In this paper, we study certificates that guarantee that the system~(\ref{Def:DiscreteSwitchingSystem}) is stable \emph{under arbitrary switching}, i.e. there exists a function $\alpha$ of class $\cK_\infty$\footnote{A function $\alpha : \R_{\geq 0} \to \R_{\geq 0}$ is of class $\cK_\infty$ ($\alpha \in \cK_\infty$) if it is continuous, $\alpha(0) = 0$, strictly increasing and unbounded.} such that for any switching signal $ \sigma : \N \to \{1, \dots, M\}$,
\[ \forall x(0) \in \R^n,  \forall k \in \N:~ \left\| x(k)\right\|~\leq~ \alpha \left(  \left\| x(0) \right\| \right). \]

One of the possible ways to assess the stability of switching systems is to use the \emph{Lyapunov theory}, and the \emph{common Lyapunov functions} (CLF) in particular. This approach consists in finding a single positive definite function that decreases along any dynamics of the system. The \emph{template}, i.e. the set in which the candidate CLF is searched, has evolved over time. One popular approach considering quadratic functions has been generalized, for example, by considering sum-of-square polynomials~\cite{AJSOS:18}, polyhedral Lyapunov functions \cite{BM1999} and then the max-min of quadratics \cite{GTHL2006}. Although the existence of a CLF is a necessary and sufficient condition for stability (see for example \cite{Jung09} and \cite{KT:2004} for the non-linear case), it is largely offset by the computing complexity required by the \virgolette{search} of this Lyapunov function. See for instance, the discussions provided in \cite{AJSOS:18}. Therefore, the \emph{Multiple Lyapunov functions} approach stands out as a promising alternative, as introduced in \cite{BRa:1998},~\cite{GoeHu06} and~\cite{Lib03} for instance. This approach aims to find (rather than a single function) a \emph{set of Lyapunov functions} whose \emph{joint decrease behaviour} provides a stability certificate. Motivated by the growing popularity of these techniques, Ahmadi et al. \cite{AJPR:14} introduced the unifying notion of \emph{path-complete Lyapunov functions} (PCLF) for which the multiple Lyapunov inequalities are encoded by the edges of a directed and labeled graph. Formally, the PCLF framework involves both combinatorial and algebraic components: first, a graph that describes the set of Lyapunov inequalities and that must be \emph{path-complete} in the sense that it captures every finite switching sequence, and then a set of candidate Lyapunov functions, called a \emph{template}, among which a solution is sought. 

The \emph{path-complete Lyapunov functions} framework provides new guidelines for constructing stability algorithms but it opens new questions and challenges, both from a theoretical and computational point of view. Indeed, the theory allows to use different graphs and different templates of functions, and thus provides a wide range of possibilities. However, it is not well understood yet when one of these algorithms provides less conservative stability certificates than another one, which has led to the problem of \emph{comparing} different path-complete graphs. More precisely, a graph is said to be \virgolette{better} than another one when its decay rate approximation capabilities surpass those of the other graph (in a sense that we will make precise in Definition~\ref{Def:order_relation_between_graphs} below). Some comparison techniques
have already been proposed (see \cite{AJPR:14}, \cite[Definition IV.2.]{PhiAth19}), but they only apply to very particular settings. In \cite{PJ:19}, a complete combinatorial characterization is proposed in a generic setting, and thus without relying on any particular property of the chosen template of candidate Lyapunov functions. On the other hand it has been observed, since the introduction of this framework~\cite{AJPR:14}, that the order relations between path-complete stability criteria strongly depend on the chosen set of candidate Lyapunov functions. On another note, sufficient conditions \cite{PEDJ:16} have already been provided in the context of \emph{constrained switching systems}, and they rely on combinatorial operations on graphs called \emph{lifts} that maintain the path-completeness. While these previous abstract lifts were introduced in order to reduce the conservatism of the arising stability condition, in this paper instead we understand how these tools can be used to characterize the comparison of path-complete graphs in the sense that all the known comparison relations can be expressed in terms of lift. In this work, moreover, we merge previous formalisms in a unifying framework based on graph operations whose validity exploits the template and the dynamics properties.

More specifically, in this paper, we propose a systematic way to compare different path-complete \virgolette{stability certificates}, based on the notion of lifts. In order to explicitly exploit the analytical properties of the chosen template and dynamics, we introduce new abstract lifts related with the aforementioned properties, providing further insight for the comparison problem. As a first application setting, we consider \emph{positive switching systems}, i.e. systems that leave the positive cone $\R^n_{\geq 0}$ invariant. In the linear case, the stability analysis of this kind of systems is characterized by their \emph{joint spectral radius} (JSR) whose approximation has been studied  for years (see \cite{Jung09} for a survey). In this paper, we consider the templates of  \emph{primal} and \emph{dual linear copositive functions} (already considered in \cite{MasSho07}), studying their analytical properties and the corresponding lifts. As final by-product of our techniques, we propose a new hierarchy of linear programs (based on path-complete Lyapunov certificates using primal and dual norms templates) in order to approximate the JSR of a set of nonnegative matrices up to an arbitrary accuracy.

The rest of this manuscript is organized as follows. First, we recall the main ideas of the path-complete Lyapunov functions framework, and we introduce the order relations among path-complete graphs. Then, we tackle the problem of their comparison thanks to the introduction of the notion of \emph{lifts} and we discuss their validity. We introduce two classes of lifts: the \emph{template-dependent lifts} as the $T$-sum lift and the min/max lifts for instance, and the \emph{template and dynamics-dependent lifts} whose validity depends on the template and dynamics properties. In Section~\ref{sec:Positive}, we focus on the family of \emph{positive linear switching systems} and the template of \emph{copositive linear functions} to illustrate the theory of the previous sections, and we develop a new numerical hierarchy to approximate the JSR. Finally, we provide a numerical example illustrating our results. Important properties from convex analysis (taken from \cite{RockConv}) are summarized, without proof, in~\ref{Sec:Appendix}. \\

\noindent {\bf Notation:} Given $M \in \N$, we denote $\langle M \rangle := \{1, \hdots, M\}$. Given $n\in \N$, $\cC^0(\R^n,\R^n)$ denotes the set of continuous vector fields on $\R^n$, while we denote by $\cC^0_+(\R^n,\R)$ the set of continuous, positive definite and radially unbounded functions. The set $\{\mathbf{e}_i\}_{i\in \langle n \rangle}$ is the canonical basis of $\R^n$.

\section{Preliminaries}

The path-complete Lyapunov functions framework generalizes previous Lyapunov techniques for discrete-time switching systems, see \cite{AJPR:14} and \cite{PhiAth19} for a thorough discussion on this topic. In what follows we briefly recall the main definitions and statements. The combinatorial structure of the Lyapunov inequalities is encoded in this setting, in a directed and labeled graph $\cG = (S,E)$ on $\langle M \rangle$ where $S$ is the finite set of nodes and $E \subseteq S \times S \times \langle M \rangle$ is the set of labeled edges. The crucial property to have effective stability criteria is then defined in the following statement.

\begin{defn}[Path-complete graph]\label{defn:Path-Completeness}
Given $M\in \N$, a graph $\cG=(S,E)$ is \emph{path-complete} on $\M$ if, for any $K \geq 1$ and any sequence $\sigma = (j_1\,\dots\,j_K)\in \M^K$, there exists a \emph{path}  $\{(a_k,a_{k+1},j_k)\}_{k=1, \dots,K}$ such that $(a_k,a_{k+1},j_k)\in E$, for each $1\leq k\leq K$.
\end{defn}

Note that the most trivial path-complete graph is the \emph{common Lyapunov function graph}, denoted by $\cG_0$, with one node and as many loops as the number of modes $M$, i.e. $\cG_0:= \left( \{a\},\{ (a,a,i) \mid i \in \langle M \rangle \}\right)$. Given a graph $\cG=(S,E)$, its \emph{dual graph} $\cG^\top=(S', E')$ is defined by $S'=S$ and $ (a,b,j)\in E \;\;\Leftrightarrow\;\;(b,a,j)\in E'$, i.e., reversing the direction of each edge. One easily verifies that a graph $\cG$ is path-complete if and only if its dual graph $\cG^\top$ is path-complete. Otherwise, there exist well-known algorithms \cite{HMU2001} to check whether a graph is path-complete; for some particular classes of graphs, it is obvious as for instance complete and co-complete graphs.

\begin{defn}[Complete and co-complete graphs]\label{def:ComplateGraph}
A graph $\cG=(S,E)$ on the alphabet $\langle M \rangle$ is \emph{complete} if for all $a \in S$, for all $i \in \langle M \rangle$, there exists at least one node $b \in B$ such that the edge $(a,b,i) \in E$. The graph is \emph{co-complete} if for all $b \in S$, for all $i \in \langle M \rangle$, there exists at least one node $a \in S$ such that the edge $(a,b,i) \in E$.
\end{defn}

The purpose of path-complete graphs stems from the following definition where the edges of a graph encode inequalities, and the path-completeness ensures that all the switching signals are covered.

\begin{defn}[Path-complete Lyapunov function]\label{defn:PCLF}
Given a switching system $F = \{ f_1, \dots, f_M\}\subset \cC^0(\R^n, \R^n)$ of dimension $n \in \N$, a \emph{path-complete Lyapunov function} (PCLF) for $F$ is a pair $(\cG,V)$ where $\cG = (S,E)$ is a path-complete graph, and $V = \{ V_s \mid s \in S \} \subseteq \cC^0_+(\R^n ,\R)$  such that the following inequalities are satisfied: 
\begin{equation}
\forall \, (a,b,i) \in E, \: \forall x \in \mathbb{R}^n: \: V_b(f_i(x))~\leq~V_a(x).
\label{eq:LyapunovInequalities}
\end{equation}

\noindent If this is the case, we say that $V$ is \emph{admissible for $\cG$ and $F$}, and we denote it by $V\in PCLF(\cG,F)$.
\end{defn}

\begin{thm}[Theorem 2.5 in \cite{PJ:19}] \label{Thm:PCLFimpliesStability}
Consider a discrete-time switching system defined by $F = \{ f_1, \dots, f_M\}\subset \cC^0(\R^n, \R^n)$. If there exists a path-complete Lyapunov function $(\cG,V)$ for $F$, then the switching system is stable.
\end{thm}

Given a switching system (\ref{Def:DiscreteSwitchingSystem}), Theorem~\ref{Thm:PCLFimpliesStability} states that the existence of a path-complete Lyapunov function is a sufficient condition for stability (see \cite{AJPR:14} and \cite{PJ:19} for the formal proof). Note that we develop here the theory for stability because we are working in a general nonlinear setting. In the linear case i.e., considering a set of linear subsystems defined by $\cA = \{ A_i \mid i \in \langle M \rangle \} \subset \R^{n \times n}$, this framework also provides estimates on the \emph{joint spectral radius (JSR)} of $\cA$, defined by
\begin{equation} \label{Eq:JSR}
\rho(\cA) ~ := ~ \lim_{k \to \infty} \max_{\overline i \in \langle M \rangle^k} \left\| A_{\overline i_k} \cdots A_{\overline i_2} A_{\overline i_1} \right\|^{1/k},
\end{equation}
which represents the best decay rate of the linear switching system defined by $\cA$. It is well-known that in this case, the stability is characterized by the JSR and it amounts to requiring that $\rho(\cA) \leq 1$. Unfortunately, this question is undecidable. For this reason, one usually focuses on computing upper (and lower) bounds of the JSR. See \cite{Jung09} for a complete discussion on this topic. \\

The path-complete Lyapunov functions framework generates a wide range of Lyapunov stability certificates since it provides two degrees of freedom: the path-complete graph $\cG$ and the template $\cV$ of candidate Lyapunov functions. Formally, we define a \emph{template} as a family of countably many sets of Lyapunov functions of fixed dimension, i.e. 
\[ \cV~:=~\bigcup_{n \in \N} \cV_n\]
where $\cV_n \subseteq \cC^0_+(\R^n,\R)$. This definition allows us to consider classic Lyapunov functions for instance, as the quadratic ones. In this case, the set $\cV_n$ for $n \in \N$ will contain all the quadratic functions $V(x) = x^\top P x$ with $P \in \R^{n \times n}$ positive definite. In what follows, we introduce order relations among path-complete graphs, formalizing the idea that one graph \virgolette{provides less conservative stability certificates} with respect to another.

\begin{defn}[Order relation between graphs] Consider two path-complete graphs $\cG$ and $\widetilde{\cG}$ on $\M$, a set of candidate Lyapunov functions $\mathcal{V}$ (a template) and a family $\mathcal{F}$ of $M$-tuples of continuous vector fields. 
\begin{enumerate}[label=(\alph*)]
\item We say that 
\begin{equation} 
\cG~\leq_{\mathcal{V}, \mathcal{F}}~\widetilde{\cG}
\label{Def:G1_smaller_G2_V_F_fixed}
\end{equation}
if, for any $F \in \mathcal{F}$, 
\begin{equation} \label{Eq:VimpliesW}
\left[ \exists V\subseteq \cV \text{ s.t. }V\in PCLF(\cG,F)\right]\: \Rightarrow \: \left[\exists W\subseteq \cV\text{ s.t. }W\in PCLF(\widetilde{\cG},F)\right].
\end{equation}
\item We say that
\begin{equation}
\cG~\leq_{\mathcal{V}}~\widetilde{\cG}
\label{Def:G1_smaller_G2_V_fixed}
\end{equation}
if the inequality (\ref{Def:G1_smaller_G2_V_F_fixed}) is satisfied for $\cF=\bigcup_{n\in \N} \cC^0(\R^n,\R^n)^M$.
\item We say that 
\begin{equation}
\cG~\leq~\widetilde{\cG}
\label{Def:G1_smaller_G2}
\end{equation}
if for any template $\mathcal{V}$, the inequality (\ref{Def:G1_smaller_G2_V_fixed}) is satisfied.
\end{enumerate}
\label{Def:order_relation_between_graphs}
\end{defn}

\begin{prop}\label{Prop:InegalityHoldsForSubgraph}
Consider two path-complete graphs $\cG$ and $\widetilde{\cG}$, a template $\cV$ and a family $\cF$ such that $\cG~\leq_{\cV,\cF}~\widetilde{\cG}$. Then, for any path-complete component $\cG'$ of $\widetilde{\cG}$, 
\[ \cG~\leq_{\cV,\cF}~\cG'. \]
The same result holds for the relations~(\ref{Def:G1_smaller_G2_V_fixed}) and (\ref{Def:G1_smaller_G2}).
\end{prop}

The proof of Proposition~\ref{Prop:InegalityHoldsForSubgraph} is straightforward, since by definition the set of inequalities encoded by $\cG'$ is a subset of the inequalities encoded by $\widetilde{\cG}$. A second property follows directly from Definition~\ref{Def:order_relation_between_graphs} and involves the common Lyapunov function graph $\cG_0$.

\begin{prop}\label{Prop:InequalityCLFGraph}
Consider a path-complete graph $\cG$ and the common Lyapunov function graph $\cG_0$. Then, for any template $\cV$ and any family $\cF$, 
\[ \cG_0~\leq_{\cV,\cF}~\cG.\] 
The same result holds for the relations~(\ref{Def:G1_smaller_G2_V_fixed}) and (\ref{Def:G1_smaller_G2}).
\end{prop}

%[TO MODIFY] It should be noted that Definition~\ref{Def:order_relation_between_graphs} does not expressly depend on the dimension of the dynamical systems. According to the definitions, the relations (\ref{Def:G1_smaller_G2_V_fixed}) and (\ref{Def:G1_smaller_G2}) hold for any dimension but these order relations might vary if we restrict the dimension of the dynamical systems. For example, one can easily see that all the path-complete graphs are equivalent in dimension $1$.\\

In \cite[Theorem 3.5]{PJ:19}, a complete characterization of the general order relation (\ref{Def:G1_smaller_G2}) is provided and relies on the combinatorial tool of \emph{simulation}\footnote{A graph $\cG = (S,E)$ \emph{simulates} a graph $\widetilde{\cG} = (\widetilde{S},\widetilde{E})$ if there exists a function $R : \widetilde{S} \to S$ such that $\forall(a,b,i) \in \widetilde{E}: (R(a),R(b),i) \in E$.}. 

\begin{thm}[Theorem 3.5 in \cite{PJ:19}] \label{Thm:Simulation}
Consider two path-complete graphs $\cG = (S,E)$ and $\widetilde{\cG} = (\widetilde{S},\widetilde{E})$. The following statements are equivalent: 
\begin{enumerate}
\item[(1)] $\cG$ simulates $\widetilde{\cG}$.
\item[(2)] $\cG \leq \widetilde{\cG}$ in the sense of Definition~\ref{Def:order_relation_between_graphs}(c).
\end{enumerate}
\end{thm}

As it has been discussed in \cite{PJ:19}, Theorem~\ref{Thm:Simulation} states that the general order relation (\ref{Def:G1_smaller_G2}) in Definition~\ref{Def:order_relation_between_graphs} is associated to a combinatorial property, notably the simulation. However, when it is not possible to establish a simulation relation, i.e. when there exist at least one template $\cV$ and one family $\cF$ such that the inequality (\ref{Def:G1_smaller_G2_V_F_fixed}) is not satisfied, it might still be possible to compare graphs with the relations~(\ref{Def:G1_smaller_G2_V_F_fixed}) and (\ref{Def:G1_smaller_G2_V_fixed}). In practice, this can result in wiser choices of template for the stability analysis in the sense of Definition~\ref{Def:order_relation_between_graphs}.

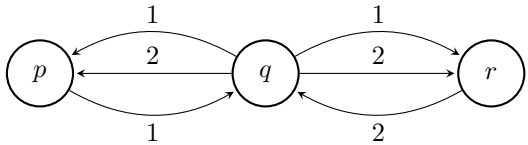
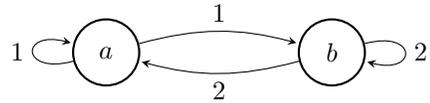
\begin{figure}[b!]
\begin{subfigure}{0.55\linewidth}
  \color{black}
  \centering
  \begin{tikzpicture}%
  [>=stealth,
  shorten >=1pt,
  node distance=1cm,
  on grid,
  auto,
  every state/.style={draw=black, fill=white,thick}
  ]
  \node[state] (node1)                 {$p$};
  \node[state] (node2) [right=of node1, xshift=2cm] {$q$};
  \node[state] (node3) [right =of node2, xshift=2cm]{$r$};
  \path[->]
  %   FROM       BEND/LOOP           POSITION OF LABEL   LABEL   TO
  (node1) edge[bend right=30]     node[below]                      {1} (node2)
  (node2) edge[bend right=30]     node[above]                      {1} (node1)
  (node2) edge[bend right=0]     node[above]                      {2} (node1)
    (node2) edge[bend left=30]     node[above]                      {1} (node3)
    (node2) edge[bend right=0]     node[above]                      {2} (node3)
    (node3) edge[bend left=30]     node[below]                      {2} (node2)
;
  \end{tikzpicture}
  \caption{The graph $\cG_1 = (S_1,E_1)$ in Example~\ref{Ex:LimitsSimulation}}
  \label{Fig:ExLimitsSimulationG1}
  \end{subfigure}
  \begin{subfigure}{0.44\linewidth}
  \color{black}
  \centering
  \begin{tikzpicture}%
  [>=stealth,
  shorten >=1pt,
  node distance=1cm,
  on grid,
  auto,
  every state/.style={draw=black, fill=white,thick}
  ]
  \node[state] (node1)         {$a$};
  \node[state] (node2) [right=of node1, xshift=2cm] {$b$};
  \path[->]
  %   FROM       BEND/LOOP           POSITION OF LABEL   LABEL   TO
  (node1) edge[loop left=180]     node                      {1} (node1)
  (node1) edge[bend left=15]     node                      {1} (node2)
  (node2) edge[bend left=15]     node                      {2} (node1)
  (node2) edge[loop right=0]     node                      {2} (node2)
;
  \end{tikzpicture}
  \caption{The graph $\cG_2 = (S_2,E_2)$ in Example~\ref{Ex:LimitsSimulation}}
\label{Fig:ExLimitsSimulationG2}
  \end{subfigure}
  \caption{Graphs of Example~\ref{Ex:LimitsSimulation}}
\end{figure}

\begin{exmp}[] \label{Ex:LimitsSimulation}
Consider the graphs $\cG_1=(S_1,E_1)$ and $\cG_2=(S_2,E_2)$ in Figures~\ref{Fig:ExLimitsSimulationG1} and \ref{Fig:ExLimitsSimulationG2} respectively. One can easily verify that $\cG_1$ does not simulate $\cG_2$. Indeed, we cannot define a relation $R:S_2 \to S_1$ with $R(a) \in S_1$ such that the edge $(R(a),R(a),1) \in E_1$ since $\cG_1$ does not have any loop. By Theorem~\ref{Thm:Simulation}, it means that $\cG_1 \nleq \cG_2$ in the sense of Definition~\ref{Def:order_relation_between_graphs}. However, one can easily prove that for any template $\cV$ closed under addition (as we will formally define in Definition~\ref{Def:ClosurePropertiesTemplate}), the inequality 
\[ \cG_1~\leq_{\cV}~\cG_2\]
holds. Indeed, let $\{V_p,V_q,V_r\} \subseteq \cV$ be admissible for $\cG_1$ and a given switching system $F$. Define the Lyapunov functions $W_a := V_p + V_q$ and $W_b := V_q + V_r$. One can easily prove that the set $\{W_a,W_b\} \subseteq \cV$ is admissible for $\cG_2$ and $F$. For example, the Lyapunov inequality 
\[ \forall x \in \R^n:~W_b(f_1(x)):=V_q(f_1(x))+V_r(f_1(x))~\leq~V_p(x)+V_q(x):=W_a(x),\]
encoded by the edge $(a,b,1) \in E_2$, holds because the Lyapunov inequalities encoded by the edges $(p,q,1)$ and $(q,r,1) \in E_1$ are satisfied by the functions $\{V_p,V_q,V_r \}$ by assumption. It implies in particular that the inequality holds for the quadratic Lyapunov functions.
\hfill $\triangle$
\end{exmp}

In this work, we want to investigate these situations and we focus our study on the two other inequalities (\ref{Def:G1_smaller_G2_V_F_fixed})  and (\ref{Def:G1_smaller_G2_V_fixed}). We introduce new combinatorial tools to understand the relation between the closure properties of the Lyapunov functions and the analytical properties of a class of systems, and the conservatism of a path-complete policy with respect to the others. 

\begin{assum}\label{Assum:SmallestGraph}
The path-complete graphs considered herein have one strongly connected component and are such that if we remove any edge, the graph is not path-complete. \end{assum}

This in particular implies that all the nodes admit at least one incoming edge and one outgoing edge. This assumption is not restrictive since our aim is to compare stability conditions: we suppose that the inequalities of the form~(\ref{eq:LyapunovInequalities}) encoded in the graphs are sufficient conditions for stability (path-completeness) without having redundant/unnecessary inequalities.

\section{Lifts}\label{sec:Lifts}
We develop in this section several expansions of graphs, called \emph{lifts}. The goal of a (valid) lift is to generate a better graph, in the sense of Definition~\ref{Def:order_relation_between_graphs}.

\begin{defn}[Lift]
Given $M \in N$, we denote with $Graphs_M$ the set of directed and labeled graphs on $\langle M \rangle$. A function $L : Graphs_M \to Graphs_M$ is a \emph{lift} if for any path-complete graph $\cG$, $L(\cG)$ is path-complete.
\end{defn}

Some examples of lifts have already been introduced  \cite{PEDJ:16} in the path-complete Lyapunov framework with the aim of improving the accuracy of the stability criteria but without exploiting the particular properties of the considered candidate Lyapunov functions template. In our case, instead, we want to use them as tools to provide a further insight about the order relations in Definition~\ref{Def:order_relation_between_graphs}, and in particular Definitions~\ref{Def:order_relation_between_graphs}(a) and~\ref{Def:order_relation_between_graphs}(b). Thus we have the following definitions.

\begin{defn}[Valid lift]\label{Def:ValidityLifts}
We say that a lift $L: Graphs_M \to Graphs_M$ is:
\begin{enumerate}[label=(\alph*)]
\item \emph{valid with respect to a template $\cV$ and a family $\cF$} if for any path-complete graph $\cG$,
\[\cG~\leq_{\cV,\cF}~L(\cG).\]
\item \emph{valid with respect to a template $\cV$} if for any path-complete graph $\cG$,
\[\cG~\leq_{\cV}~L(\cG).\]
\item \emph{valid} if for any path-complete graph $\cG$,
\[ \cG~\leq~L(\cG). \]
\end{enumerate}
\end{defn}

\noindent To be consistent with Theorem~\ref{Thm:Simulation} that characterizes the general inequality~(\ref{Def:G1_smaller_G2}), a lift is valid if and only if there exists a simulation relation between $\cG$ and $L(\cG)$. This is the case for both $T$-product and $M$-path-dependent lifts defined in \cite{PEDJ:16}, for instance. In this work, we are particularly interested in the order relations~(\ref{Def:G1_smaller_G2_V_F_fixed}) and (\ref{Def:G1_smaller_G2_V_fixed}) and therefore, we focus our study on lifts that are valid with respect to a template (and a family) as in Definition~\ref{Def:ValidityLifts}(a) and (b). Indeed, quadratic functions are \emph{closed under addition}. By this, we mean that the sum of two quadratic functions of a fixed dimension can also be expressed as a quadratic function. It turns out that this property is key for the relation (\ref{Def:G1_smaller_G2_V_fixed}) (see Theorem~\ref{Thm:GeneralThmLifts} below). More generally, we will show that such a closure property allows us to define lifts that are valid in a specific setting, even though they are not valid in general (i.e. in the sense of Definition~\ref{Def:order_relation_between_graphs}).

\subsection{Template-dependent lifts}

In this section, we study the consequences of closure properties of the template on the path-complete stability certificates, and we focus on lifts whose nodes of the lifted graph are associated to subsets of the initial set of nodes. In what follows, we introduce three \emph{template-dependent lifts}, that are lifts whose validity depends on the template properties. 

%\begin{defn}[Template closed under a family of binary operations]
%Given a template $\cV$ of candidate Lyapunov functions and a family of binary operations $\{\star_n:\cC^0_+(\R^n,\R)\times\cC^0_+(\R^n,\R)\to \cC^0_+(\R^n,\R)\}_{n \in \N}$. We say that the template $\cV$ \emph{is closed under $\{\star_n\}_{n\in\N}$} if for all $n \in \N$ and  for all $f_1, f_2 \in \cV_n$, $f_1 \star_n f_2 \in \cV_n$.
%\end{defn}

\begin{defn}[Closure properties of a template] \label{Def:ClosurePropertiesTemplate}
Given a template $\cV = \cup_{n \in \N} \cV_n $ of candidate Lyapunov functions and a family of binary operations $\{\star_n:\cC^0_+(\R^n,\R)\times\cC^0_+(\R^n,\R)\to \cC^0_+(\R^n,\R)\}_{n \in \N}$.
\begin{enumerate}
    \item[(a)] For a fixed dimension $n \in \N$, we say that the set of functions $\cV_n$ is \emph{closed under the binary operation $\star_n$} if for all $f_1, f_2 \in \cV_n$, $f_1 \star_n f_2 \in \cV_n$.
    \item[(b)] We say that the \emph{template $\cV$ is closed under the family of binary operations $\{\star_n\}_{n \in \N}$ is for all $n \in \N$, the set $\cV_n$ is closed under $\star_n$. }
\end{enumerate}
\end{defn}

As first and remarkable example of binary operation, we consider the law of \emph{addition} under which many usual templates are closed, as, for example, quadratic functions, convex functions or sum-of-squares polynomials. Therefore, in what follows, we define the \emph{$T$-sum lift}, which explores the existing relations between \emph{sums} of $T$ functions/nodes of the initial graph. To this aim, given a set $S$ and $T\in N$, we denote with $Multi^T(S)$ the set of multi-sets of cardinality $T$ with elements in $S$ , where a multi-set is defined as a set with possible repetitions.

\begin{defn}[$T$-sum lift]
Given $T\in \N$ and a graph $\cG = (S,E)$ on the alphabet $\langle M \rangle$, the \emph{$T$-sum lift}, denoted by $\cG^{\oplus T} = (S^{\oplus T},E^{\oplus T})$, is defined as follows :

\begin{enumerate}
\item[(1)] The set of nodes $S^{\oplus T}$ is defined by 
\[ S^{\oplus T}~:=~Multi^T(S) .\]

\item[(2)] For each multi-set of edges of $E$ of the form $\left \{(a_1,b_1,i), \dots,(a_T,b_T,i)\right \}$ with $i \in \langle M \rangle$ such that $\{a_1,\dots, a_T\}$ and $\{b_1, \dots,  b_T\} \in S^{\oplus T}$, the edge $(\{a_1, \dots, a_T\}, \{b_1, \dots, b_T\},i) \in E^{\oplus T}$.
\end{enumerate}
\label{Def:SumLift}
\end{defn}

Even though the binary operation of the sum is natural and many templates of Lyapunov functions are closed under sum, it turns out that our approach generalizes to less straightforward binary operations. For instance, the templates of piecewise $\cC^1$ functions \cite{dellarossa2020piecewise} and polyhedral functions \cite{AJ2020}, which are usually used for stability analysis, are closed under pointwise maximum of finitely many functions \cite[Proposition 3]{dellarossa2020piecewise}. These results motivate the introduction of both \emph{min} and \emph{max lifts}.

\begin{defn}[Max/Min-lift]\label{defn:MinLift} Given a graph $ \cG = (S,E)$ on the alphabet $\langle M \rangle$. 

\begin{enumerate}
\item[(a)] The \emph{max lift}, denoted by $\cG_{\max} = (S_{\max} , E_{\max})$, is defined as follows:
\begin{enumerate}
\item[(1)] The set of nodes $S_{\max}$ is defined by 
\[S_{\max}~:=~\{ S' \subseteq S \mid S' \neq \emptyset \}.\]
\item[(2)] An edge $(A,B,i)\in E_{max}$ with $A, B \in S_{\max}$ and $i \in \langle M \rangle$ if and only if for all $b \in B$, there exists at least one $a \in A$ such that $(a,b,i) \in E$.
\end{enumerate}
\item[(b)] The \emph{min lift}, denoted by $\cG_{\min} = (S_{\min} , E_{\min})$, is defined as follows:
\begin{enumerate}
\item[(1)] The set of nodes $S_{\min}$ is defined by 
\[S_{\min}~:=~\{ S' \subseteq S \mid S' \neq \emptyset \}.\]
\item[(2)] An edge $(A,B,i)\in E_{min}$ with $A, B \in S_{\min}$ and $i \in \langle M \rangle$ if and only if for all $a \in A$, there exists at least one $b \in B$ such that $(a,b,i) \in E$.
\end{enumerate}
\end{enumerate}
\end{defn}

\begin{rem} \label{rem:LiftsIsomorphic}
Note that we omit to prove that the lifts are well-defined because the path-completeness follows directly. Indeed, given a path-complete graph $\cG$, we can observe that the sum-,min-,and max-lifted graphs admit a path-complete and strongly connected component isomorphic to the initial graph. See \cite{DDLJ21} for more details.
\hfill $\triangle$
\end{rem}

We state now the main theorem of this paper which discusses the validity of the lifts introduced in Definitions~\ref{Def:SumLift} and \ref{defn:MinLift}.

\begin{thm} \label{Thm:GeneralThmLifts}
Consider $T \in \N$ and a family of binary operations $\{\star_n\}_{n \in \N}$ such that $\forall n \in \N$, $\star_n$ corresponds to the addition (resp. pointwise maximum, pointwise minimum). The $T$-sum (resp. max, min) lift is valid with respect to any template closed under $\{\star_n\}_{n \in \N}$.

%Given a binary operation $\star \in \{sum, min, max \}$. The $\star$-lift is valid with respect to any template closed under $\star$.
\end{thm}

The proof for each lift will be developed in the following subsections.

\subsubsection{The T-sum lift}

In this section, we prove  Theorem~\ref{Thm:GeneralThmLifts} in the case of the $T$-sum lift introduced in Definition~\ref{Def:SumLift}.

\begin{proof}[Proof of Theorem~\ref{Thm:GeneralThmLifts}, $T$-sum lift]
Consider a path-complete graph $\cG=(S,E)$ on the alphabet $\langle M \rangle$, a template $\cV$ of candidate Lyapunov functions closed under addition and any family $\cF := \{f_i\}_{i\in \M}$. Suppose that there exists a set of functions $\{V_s \mid s \in S\} \subset \mathcal{V}$ admissible for $\cG$ and $\cF$, and for any $\overline a= \{a_1,\dots,a_T\}\in S^{\oplus T}$ define
\begin{equation} \label{def:LyapFunctionsSumLift}
    W_{\overline a}~:=~V_{a_1} + \dots + V_{a_T} \in \mathcal{V}.
\end{equation}
The Lyapunov inequalities \eqref{eq:LyapunovInequalities} of $\cG^{\oplus T}$ are satisfied because, for every edge $(\overline a, \overline b, i) \in E^{\oplus T}$, we have
\[
\begin{aligned}
W_{\overline b} \left(f_i(x)\right)  &=  \left( V_{b_1}(f_i(x)) +\dots + V_{b_T}(f_i(x)) \right),\\[0.1cm]
& \leq  \left( V_{a_1}(x) + \dots + V_{a_T}(x) \right) \\[0.1cm]
&=  W_{\overline{a}}(x),
\end{aligned} \]

\noindent for all $x \in \R^n$ since $(a_1,b_1,i)$, $\dots$,$(a_T,b_T,i)\in E$ by Definition~\ref{Def:SumLift} (possibly after a re-ordering of $\overline a$ and $\overline b$).
\end{proof}

\begin{figure}[b!]
\begin{subfigure}{0.49\linewidth}
  \color{black}
  \centering
  \begin{tikzpicture}%
  [>=stealth,
  shorten >=1pt,
  node distance=1cm,
  on grid,
  auto,
  every state/.style={draw=black, fill=white,thick}
  ]
  \node[state] (left)                  {$a$};
  \node[state] (right) [right=of left, xshift=3cm] {$b$};
  \path[->]
  %   FROM       BEND/LOOP           POSITION OF LABEL   LABEL   TO
  (left) edge[loop left=60]     node                      {2} (left)
  (left) edge[bend left=15]     node                      {1} (right)
  (right) edge[bend left=15]     node                      {1} (left)
  (right) edge[loop right=300]     node                      {2} (right)
;
  \end{tikzpicture}
  \caption{The graph $\cG_3 = (S_3,E_3)$ in Example~\ref{ex:Sum}}
  \label{Fig:InitialSumLift}
  \end{subfigure}
  \begin{subfigure}{0.49\linewidth}
  \color{black}
  \centering
  \begin{tikzpicture}%
  [>=stealth,
  shorten >=1pt,
  node distance=1cm,
  on grid,
  auto,
  every state/.style={draw=black, fill=white,thick}
  ]
  \node[state] (left)  [yshift=1.5cm]                {$\{a,a\}$};
  \node[state] (right) [right=of left, xshift=3cm] {$\{b,b\}$};
  \node[state] (below) [below right=of left, xshift=1.4cm, yshift=-1cm]{$\{a,b\}$};
  \path[->]
  %   FROM       BEND/LOOP           POSITION OF LABEL   LABEL   TO
  (left) edge[loop left=60]     node                      {2} (left)
  (left) edge[bend left=15]     node                      {1} (right)
  (right) edge[bend left=15]     node                      {1} (left)
  (right) edge[loop right=300]     node                      {2} (right)
  (below) edge[loop right=300]     node                      {2} (below)
  (below) edge[loop left=60]     node                      {1} (below)
;
  \end{tikzpicture}
  \caption{The 2-sum-lift $(\cG_3)^{\oplus 2}$}
\label{Fig:SumLiftedGraph}
  \end{subfigure}
  \caption{Example of a $2$-sum-lifted graph}
\end{figure}
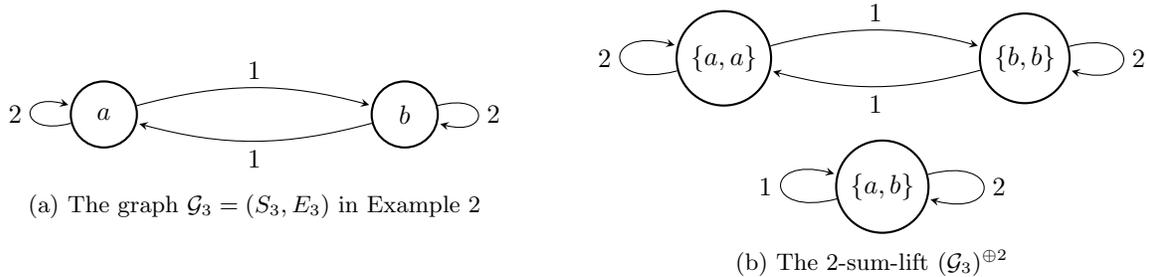

\begin{exmp}\label{ex:Sum}
Consider the path-complete graph $\cG_3 = (S_3,E_3)$ on the alphabet $\langle M \rangle := \{1,2\}$ in Figure~\ref{Fig:InitialSumLift}, and apply the $2$-sum lift in Definition~\ref{Def:SumLift} to $\cG_3$. The outcome is provided in  Figure~\ref{Fig:SumLiftedGraph}. As expected, the lifted graph $(\cG_3)^{\oplus 2}$ admits three nodes, one for each multi-set of cardinality $2$ of the initial set of nodes $S_3 = \{a,b\}$, i.e. $Multi^2(S_3) = \left\{ \{a,a\}, \{a,b\}, \{b,b\} \right\}$. By Theorem~\ref{Thm:GeneralThmLifts}, we know that for any template $\cV$ closed under sum, the inequality 
\[ \cG_3~\leq_{\cV}~(\cG_3)^{\oplus 2}\]
holds. By Proposition~\ref{Prop:InegalityHoldsForSubgraph}, this inequality is also verified for the two path-complete and strongly connected components of $(\cG_3)^{\oplus 2}$. As reported by Remark~\ref{rem:LiftsIsomorphic}, one of the components induced by the nodes $\{a,a\}$ and $\{b,b\}$ is isomorphic to the graph $\cG_3$ itself. The second one induced by the node $\{a,b\}$ is isomorphic to the common Lyapunov function graph $\cG_0$ since the node associated to $\{a,b\}$ admits one loop for each mode. So, Theorem~\ref{Thm:GeneralThmLifts} and Proposition~\ref{Prop:InegalityHoldsForSubgraph} imply together that 
\[ \cG_3~\leq_{\cV}~\cG_0\]
for any template $\cV$ closed under sum. Moreover, by Proposition~\ref{Prop:InequalityCLFGraph}, we know that the reverse inequality holds for any template and any switching system. In particular, 
\[ \cG_0~\leq_{\cV}~\cG_3\]
for any template $\cV$ closed under sum. We have thus proved that the graphs $\cG_3$ and $\cG_0$ are equivalent in the sense of the order relation~(\ref{Def:G1_smaller_G2_V_fixed}) for any template closed under sum. In practice, it means that given such a template $\cV$ and a switching system $F$, either both graphs $\cG_3$ and $\cG_0$ admit a solution admissible for $\cV$ and $F$, or none of them. That is, the inequalities encoded in $\cG_3$ are as conservative as the ones encoded in $\cG_0$.\hfill $\triangle$
\end{exmp}

\begin{rem}Note that the $T$-sum lift introduced in Definition~\ref{Def:SumLift} is a  generalization of the construction presented in \cite{PhiAth19}. Indeed, the comparison (\ref{Def:G1_smaller_G2_V_F_fixed}) of path-complete graphs in Definition~\ref{Def:order_relation_between_graphs} is tackled for the particular template of quadratic functions (closed under sum) and the linear switching systems. A sufficient condition to the implication (\ref{Eq:VimpliesW}) is provided and consists in checking whether the solution ($W$) of the second graph can be defined as a conic combination of the solution ($V$) of the first graph regardless of the switching system (see \cite[Definition IV.2.]{PhiAth19} for the formal definition).
\hfill $\triangle$
\end{rem}

\subsubsection{The min and max lifts}

We provide here the proof of Theorem~\ref{Thm:GeneralThmLifts} in the case of max and min lifts introduced in Definition~\ref{defn:MinLift}.

\begin{proof}[Proof of Theorem~\ref{Thm:GeneralThmLifts}, max lift] 
Consider a template $\cV$ closed under pointwise maximum and any family of vector fields $\{f_i\}_{i\in \M}$. Suppose that there exists a PCLF for the initial graph $\cG$ of the form $\{V_s \mid s \in S\} \subset \mathcal{V}$. Given any $A\in S_{\max}$ the corresponding Lyapunov function $W_A\in \cV$ is defined by
\begin{equation}
    \forall x \in \R^n:~W_{A}(x)~:=~\underset{a\in A}{\max} \, V_{a}(x).
    \label{def:LyapFunctionsMaxLift}
\end{equation}
Given $(A,B,i)\in E_{\max}$,  we have 
\[W_B(f_i(x))~=~\max_{b\in B} V_{b}(f_i(x))~\leq~\max_{a\in A}V_a(x)~=~W_A(x), \]
for any $x\in \R^n$, since, by Definition~\ref{defn:MinLift}, for all $b\in B$ there exists at least a $a\in A$ such that $V_b(f_i(x))\leq V_a(x)$, concluding the proof.
\end{proof}

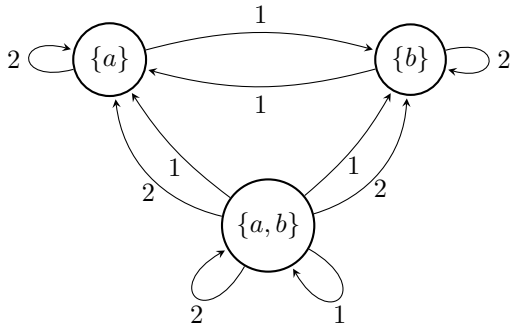
\begin{figure}[b!]
  \color{black}
  \centering
  \begin{tikzpicture}%
  [>=stealth,
  shorten >=1pt,
  node distance=1cm,
  on grid,
  auto,
  every state/.style={draw=black, fill=white, thick}
  ]
  \node[state] (left)  [yshift=2.5cm]                {$\{a\}$};
  \node[state] (right) [right=of left, xshift=3cm] {$\{b\}$};
  \node[state] (below) [below right=of left, xshift=1.4cm, yshift=-1.5cm]{$\{a,b\}$};
  \path[->]
  %   FROM       BEND/LOOP           POSITION OF LABEL   LABEL   TO
  (left) edge[loop left=60]     node                      {2} (left)
  (left) edge[bend left=15]     node                      {1} (right)
  (right) edge[bend left=15]     node                      {1} (left)
  (right) edge[loop right=300]     node                      {2} (right)
  (below) edge[bend left=10]     node[below]                      {1} (left)
  (below) edge[bend left=35]     node[below]                      {2} (left)
  (below) edge[bend right=10]     node[below]                      {1} (right)
  (below) edge[bend right=35]     node[below]                      {2} (right)
  (below) edge [out=330,in=300,looseness=8] node[below] {1} (below)
  (below) edge [out=240,in=210,looseness=8] node[below] {2} (below);
;
\end{tikzpicture}
\caption{The max-lift $(\cG_3)_{\max}$}
\label{Fig:ADHS_max_lift}
\end{figure}

\begin{exmp}\label{ex:MaxLIft}
Consider the path-complete graph $\cG_3 = (S_3,E_3)$ on two modes ($M = 2$) in Figure~\ref{Fig:InitialSumLift}. The graph $(\cG_3)_{\max}$ in Figure~\ref{Fig:ADHS_max_lift} results from the application of the max lift introduced in Definition~\ref{defn:MinLift} to the graph $\cG_3$. It consists in two path-complete and strongly connected components. As stated by Remark~\ref{rem:LiftsIsomorphic}, there is a subgraph isomorphic to $\cG_3$ itself, and the second one is isomorphic to the common Lyapunov function graph $\cG_0$. By Theorem~\ref{Thm:GeneralThmLifts}, the inequality 
\[ \cG_3~\leq_{\cV}~(\cG_3)_{\max} \]
holds for any template $\cV$ closed under pointwise maximum. Moreover, by Proposition~\ref{Prop:InegalityHoldsForSubgraph}, the same inequality holds for the component $\cG_0$ of $(\cG_3)_{\max}$. Since the reverse inequality holds trivially by Proposition~\ref{Prop:InequalityCLFGraph}, it means that $\cG_3$ provides stability criteria as conservative as those of $\cG_0$ for this kind of templates. \hfill $\triangle$
\end{exmp}

Actually the construction of the min lift is dual (in a sense that we will clarify) to the construction of the max lift, so the proof of Theorem~\ref{Thm:GeneralThmLifts} for the min lift follows the same structure as the proof for the max lift.

\begin{proof}[Proof of Theorem~\ref{Thm:GeneralThmLifts}, min lift]

The proof is the same as the proof for max, but given any $A \in S_{\min}$ the corresponding Lyapunov function $W_A \in \cV$ is defined by
$W_{A}(x):=\underset{a\in A}{\min} \, V_{a}(x)$, $ \forall x \in \R^n$.
\end{proof}

\begin{figure}[b!]
\begin{subfigure}{0.49\linewidth}
  \color{black}
  \centering
  \begin{tikzpicture}%
  [>=stealth,
  shorten >=1pt,
  node distance=1cm,
  on grid,
  auto,
  every state/.style={draw=black, fill=white, thick}
  ]
  \node[state] (left)                  {$a$};
  \node[state] (right) [right=of left, xshift=3cm] {$b$};
  \path[->]
  %   FROM       BEND/LOOP           POSITION OF LABEL   LABEL   TO
  (left) edge[loop left=60]     node                      {1} (left)
  (left) edge[bend left=15]     node                      {2} (right)
  (right) edge[bend left=15]     node                      {1} (left)
  (right) edge[loop right=300]     node                      {2} (right)
;
  \end{tikzpicture}
  \caption{The graph $\cG_4 = (S_4,E_4)$ in Example~\ref{ex:MinLIft}}
  \label{Fig:ADHS_initial_min_lift}
  \end{subfigure}
  \begin{subfigure}{0.49\linewidth}
  \color{black}
  \centering
  \begin{tikzpicture}%
  [>=stealth,
  shorten >=1pt,
  node distance=1cm,
  on grid,
  auto,
  every state/.style={draw=black, fill=white, thick}
  ]
  \node[state] (left)  [yshift=2.5cm]                {$\{a\}$};
  \node[state] (right) [right=of left, xshift=3cm] {$\{b\}$};
  \node[state] (below) [below right=of left, xshift=1.4cm, yshift=-1.5cm]{$\{a,b\}$};
  \path[->]
  %   FROM       BEND/LOOP           POSITION OF LABEL   LABEL   TO
  (left) edge[loop left=60]     node                      {1} (left)
  (left) edge[bend left=15]     node                      {2} (right)
  (right) edge[bend left=15]     node                      {1} (left)
  (right) edge[loop right=300]     node                      {2} (right)
  (left) edge[bend right=30]     node[below]                      {1} (below)
  (left) edge[bend right=55]     node[below]                      {2} (below)
  (right) edge[bend left=30]     node[below]                      {1} (below)
  (right) edge[bend left=55]     node[below]                      {2} (below)
  (below) edge[bend right=0]     node[above]                      {1} (left)
  (below) edge[bend left=0]     node                      {2} (right)
  (below) edge [out=330,in=300,looseness=8] node[below] {1} (below)
  (below) edge [out=240,in=210,looseness=8] node[below] {2} (below);
;
  \end{tikzpicture}
  \caption{The min-lift $(\cG_4)_{\min}$}
  \label{Fig:ADHS_min_lift}
  \end{subfigure}
  \caption{Example of a min-lifted graph}
\end{figure}
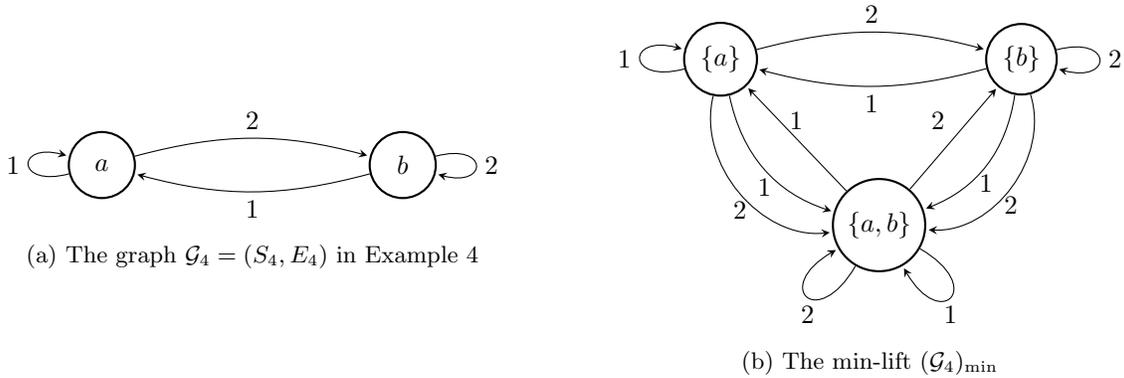

% A retirer ou pas ?
\begin{exmp}\label{ex:MinLIft}
Consider the path-complete graph $\cG_4 = (S_4,E_4)$ with $M = 2$ in Figure~\ref{Fig:ADHS_initial_min_lift}. The min lifted graph $(\cG_4)_{\min}$ illustrated in Figure~\ref{Fig:ADHS_min_lift} has 3 nodes, one for each subset of $(S_4)_{\min}$ and we can easily identify two particular strongly connected components: first, a component isomorphic to $\cG_4$ as predicted by Remark~\ref{rem:LiftsIsomorphic}, and a component isomorphic to the common Lyapunov function graph $\cG_0$. By Theorem~\ref{Thm:GeneralThmLifts} and Propositions~\ref{Prop:InegalityHoldsForSubgraph} and \ref{Prop:InequalityCLFGraph}, it means that $\cG_4$ is equivalent to $\cG_0$ for any template closed under pointwise minimum. \hfill $\triangle$
\end{exmp}

Proposition~\ref{Prop:InequalityCLFGraph} tells us that the graph associated to the common Lyapunov function $\cG_0$ is worse than any path-complete graph in the sense of Definition~\ref{Def:order_relation_between_graphs}. But one can easily see that for some graphs and under some assumptions (\cite[Theorem III.8]{PhiAth19}), the reverse inequality holds.
\begin{prop}\label{prop:CompletenessandMAxMin}
Consider a path-complete graph $\cG$ and the common Lyapunov function graph $\cG_0$.
\begin{enumerate}
    \item[(a)] If $\cG$ is \emph{complete}, $\cG_0$ is a path-complete component of $\cG_{\min}$ and thus
    \[\cG~\leq_{\cV}~\cG_0 \]
    for any template $\cV$ closed under pointwise min.
    \item[(b)] If $\cG$ is \emph{co-complete}, $\cG_0$ is a path-complete component of $\cG_{\max}$ and thus
    \[\cG~\leq_{\cV}~\cG_0 \]
    for any template $\cV$ closed under pointwise max.
    \item[(c)] The following inequalities
    \[ \cG ~\leq_{\cV}~(\cG_{\max})_{\min}~\leq_{\cV}~\cG_0\]
    and 
    \[ \cG ~\leq_{\cV}~(\cG_{\min})_{\max}~\leq_{\cV}~\cG_0 \]
    hold for any template $\cV$ closed under pointwise minimum and maximum.
\end{enumerate}
\end{prop}

%\begin{rem} 
%Moreover, it has been proved \cite[Theorem III.8]{PhiAth19} that it is always possible to define a common Lyapunov function expressed as the composition of pointwise minima and maxima of the piecesof a PCLF. In terms of lift, this theorem means that, for any path-complete graph $\cG$, the following inequalities
%\[ \virginie{\cG ~\leq_{\cV}~}(\cG_{\max})_{\min}~\leq_{\cV}~\cG_0\]
%and 
%\[ \virginie{\cG ~\leq_{\cV}~}(\cG_{\min})_{\max}~\leq_{\cV}~\cG_0 \]
%are satisfied for any template $\cV$ closed under pointwise minimum and maximum. Since the reverse inequalities hold by Proposition~\ref{Prop:InequalityCLFGraph}, all the path-complete graphs are equivalent in the sense of Definition~\ref{Def:order_relation_between_graphs} if we consider a template closed under pointwise minimum and maximum.\end{rem}

\subsection{Lifts that depend on both the template and the dynamics properties}

In this section and in order to study the order relation~(\ref{Def:G1_smaller_G2_V_F_fixed}) in Definition~\ref{Def:order_relation_between_graphs}, we introduce another lift whose validity depends on both the template and the dynamics properties. This lift sheds new light, and provides a generalization of previous results in the literature, such as \cite[Proposition 4.2]{AJPR:14}, \cite[Example IV.11]{PhiAth19} and \cite[Example 3.9]{PJ:19}. For instance, the \emph{composition lift} introduced in Definition~\ref{def:ForwardCompoLift} was implicitly used  in the particular case of quadratic Lyapunov functions and linear switching systems, but it was not clear how it could be used in a general setting. Theorem~\ref{Thm:CompositionLifts} now answers the question.

\begin{defn}[Composition lift]\label{def:ForwardCompoLift}
Given a graph $G = (S,E)$ on the alphabet $\langle M \rangle $. The composition lift, denoted by $G_{\text{co}} = (S_{\text{co}},E_{\text{co}})$, is defined as follows :
\begin{enumerate}
\item[(1)] The set of nodes $S_{\text{co}}$ is defined by 
\[ S_{\text{co}} ~ := ~ \left\{ s \circ i \mid s \in S, \, i \in \M \right\}. \]
\item[(2)] For each edge $(a,b,i) \in E$, and each mode $j \in \langle M \rangle$, the edge $(a \circ j, b \circ i,j) \in E_{\text{co}}$.
\end{enumerate}
\end{defn}

\begin{prop}\label{Prop:ForwardBackwardLiftsP-C}
The composition lift preserves the path-completeness.
\end{prop}
\begin{proof}[Proof of Proposition~\ref{Prop:ForwardBackwardLiftsP-C}]
Consider a path-complete graph $\cG = (S,E) $, its composition lifted graph $\cG_{co} = (S_{co},E_{co})$ and a switching signal $\sigma$ of length $k \in \N$. By path-completeness of $\cG$, there exist $k$ consecutive edges in $E$ denoted by $e_j \, = \, (a_j , a_{j+1}, \sigma(j)) $ where $a_j \in S$ for every $j$. For each edge $e_k$ with $k = 1, \hdots, j$, $\tilde{e}_k \, = \,  \left( a_k \sigma({k-1}), a_{k+1} \sigma(k), \sigma(k-1) \right) \in E_{co}$. Using Assumption~\ref{Assum:SmallestGraph}, there is an outgoing edge $(a_{k+1},a_{k+2},\sigma({k+1})) \in E $ from node $a_{k+1}$. Hence, the edge $(a_{k+1}\sigma(k) , a_{k+2} \sigma({k+1}), \sigma(k)) \in E_{co}$ which concludes the proof.
\end{proof}

We prove a similar result to Theorem~\ref{Thm:GeneralThmLifts} for the composition lift in Definition~\ref{def:ForwardCompoLift}. By extension of Definition~\ref{Def:ClosurePropertiesTemplate}, we say that a template $\cV = \cup_{n \in \N} \cV_n$ is \emph{closed under composition with the dynamics of $\cF$} if for any $n \in \N$, for all $V \in \cV_n$ and $f \in \cF \cap \cC^0(\R^n,\R^n)$, the composition $V \circ f \in \cV_n$.
\begin{thm}\label{Thm:CompositionLifts}
The composition lift is valid with respect to any family $\cF$ of systems and any template $\cV$ closed under composition with the dynamics of $\cF$.
\end{thm}

\begin{proof}[Proof of Theorem~\ref{Thm:CompositionLifts}]
Consider a family of systems $\cF$, a system $F = \{f_i\}_{i\in \M} \subseteq \cF$ and a template $\cV$ closed under composition with any dynamics of any system in $\cF$. Suppose that there exists a PCLF for an initial path-complete graph $\cG=(S,E)$ of the form $\{V_s \mid s \in S\} \subset \mathcal{V}$. Given $s \in S$ and $i \in \langle M \rangle$, the corresponding Lyapunov function $W_{s \circ i}$ is defined by 
\begin{equation} \label{def:LyapFunctionsForwardCompoLift}
    \forall x \in \R^n:~W_{s \circ i}(x)~:=~\left(V_s \circ f_i \right) (x).
\end{equation}
Given $(a \circ j,b \circ i,j)\in E_{co}$,  we have 
\[W_{b \circ i}(f_j(x))~=~ V_b \left( f_i  \left(f_j(x)\right) \right)~\leq~V_a \left(f_j(x)\right)~=~W_{a \circ j}(x), \]
for any $x\in \R^n$, since, by Definition~\ref{def:ForwardCompoLift},  $(a,b,i) \in E$ and the Lyapunov inequality encoded by this edge is satisfied by the Lyapunov functions $\{V_s \mid s \in S \}$, especially when they are evaluated at points of the form $y = f_j(x)$. 
\end{proof}

\begin{rem}\label{Rem:BackwardLift}
If the dynamics in a family $\cF$ are invertible, a similar lift to the composition lift in Definition~\ref{def:ForwardCompoLift}, referred as the \emph{backward composition lift}, can be defined. In this case, the Lyapunov functions associated to the nodes of the lifted graph are defined as the composition with the inverse dynamics, i.e. 
\[ \forall x \in \R^n:~W_{s \circ i}(x)~:=~\left( V_s \circ {f_i}^{-1} \right)(x). \]
In a similar way, one can prove that this lift is valid with respect to any family $\cF$ of systems with invertible dynamics and any template $\cV$ closed under composition with the inverse dynamics of $\cF$.
\hfill $\triangle$
\end{rem}

\begin{exmp}\label{ex:forwardCompositionLift}
Consider the path-complete graph $\cG_4 = (S_4,E_4)$ on $2$ modes in Figure~\ref{Fig:ADHS_initial_min_lift}. We apply the composition lift in Definition~\ref{def:ForwardCompoLift} to $\cG_4$ and we get the graph $(\cG_4)_{co}$ in Figure~\ref{Fig:CompoLift}. As expected from Definition~\ref{def:ForwardCompoLift}, the lifted graph has $4$ nodes, one for each pair $\{s,i\} \subseteq S_4 \times \langle M \rangle$. Observe that the graph obtained with the composition lift $(\cG_4)_{co}$ admits a strongly connected component isomorphic to the dual graph of $\cG_4$. By Theorem~\ref{Thm:CompositionLifts} and Proposition~\ref{Prop:InegalityHoldsForSubgraph}, it implies that 
\[ \cG_4~\leq_{\cV,\cF}~(\cG_4)^\top,\]
for any family $\cF$ of systems and any template $\cV$ closed under composition with the dynamics of $\cF$. Actually, the application of the backward composition lift in Remark~\ref{Rem:BackwardLift} to $(\cG_4)^\top$ generates a component isomorphic to $\cG_4$. This means that $(\cG_4)^\top~\leq_{\cV,\cF}~\cG_4$ for any family $\cF$ of systems of invertible dynamics and any template $\cV$ closed under composition with the inverse dynamics of $\cF$, and hence the graph $\cG_4$ and its dual $(\cG_4)^\top$ are equivalent in the sense of (\ref{Def:G1_smaller_G2_V_F_fixed}) for any family $\cF$ of systems with invertible dynamics and any template $\cV$ closed under composition with the dynamics of $\cF$ and their inverse. This result generalises~\cite[Example 3.9.]{PJ:19} where the equivalence of these graphs is proved for the invertible linear switching systems and the template of quadratic functions. Note that the proof there uses an ad hoc technique, which cannot be generalised easily to arbitrary settings. \hfill $\triangle$
\end{exmp}

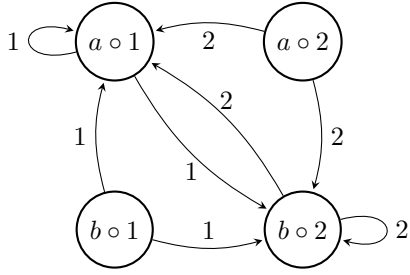
\begin{figure}[t!]
   \color{black}
  \centering
  \begin{tikzpicture}%
  [>=stealth,
  shorten >=1pt,
  node distance=1cm,
  on grid,
  auto,
  every state/.style={draw=black, fill=white, thick}
  ]
  \node[state] (left1)                 {$a \circ 1$};
  \node[state] (right1) [right=of left1, xshift=1.5cm] {$a \circ 2$};
  \node[state] (left2) [below =of left1, yshift=-1.5cm]{$b \circ 1$};
  \node[state] (right2) [below =of right1, yshift=-1.5cm]{$b \circ 2$};
  \path[->]
  %   FROM       BEND/LOOP           POSITION OF LABEL   LABEL   TO
  (left1) edge[loop left=60]     node                      {1} (left1)
  (right1) edge[bend right=15]     node                      {2} (left1)
  (left2) edge[bend left=15]     node                      {1} (left1)
  (right1) edge[bend left=15]     node                    {2} (right2)
  (right2) edge[bend right=15]     node[above]                      {2} (left1)
  (left2) edge[bend right=15]     node                     {1} (right2)
  (left1) edge[bend right=15]     node[below]                     {1} (right2)
  (right2) edge[loop right=15]     node                      {2} (right2)
;
\end{tikzpicture}
\caption{The composition lifted graph $(\cG_4)_{co}$}
\label{Fig:CompoLift}
\end{figure}

\section{Application to Positive switching Systems and Copositive Linear Norms}
\label{sec:Positive}
In this section we apply the ideas previously developed in studying stability of \emph{positive} switching systems of the form
\begin{equation}\label{eq:PosSystem}
x(k+1) ~=~A_{\sigma(k)}x(k)
\end{equation}
where $\sigma:\N\to \M$, and $\cA:=\{A_1,\dots, A_M\}\subset\R^{n\times n}$ is a set of nonnegative matrices.
We recall that a matrix $A\in \R^{n\times n}$ is said to be \emph{nonnegative} if it has entries in $\R_{\geq 0}$ ($A\in \R^{n\times n}_{\geq 0}$), \emph{positive} if it is nonnegative and nonzero, and \emph{strictly positive} if $A\in \R^{n\times n}_{>0}$. If $A\in \R^{n\times n}_{\geq 0}$ we have that $Ax\in \R^n_{\geq 0}$ for all $x\in \R_{\geq 0}$. If, moreover, $A$ is also invertible, we have $Ax\in \R^n_{>0}$ for all $x\in \R_{> 0}$. Positive switching systems~\eqref{eq:PosSystem} are popular for modeling the dynamics of phenomena constrained in the positive cone $\R^n_{\geq0}$, and from our point of view provides a simple ``practical'' setting in order to illustrate the developments of previous sections.
\subsection{Copositive norms:  closure properties, duality, and valid lifts}\label{subsec:CopNorms}
In this section we will consider and study two particular functions templates, the \emph{primal/dual linear copositive norms}. In particular, we investigate their closure properties and the duality correspondence between them, in order to apply the ideas presented in Section~\ref{sec:Lifts} to the stability analysis of positive switching systems as in~\eqref{eq:PosSystem}.

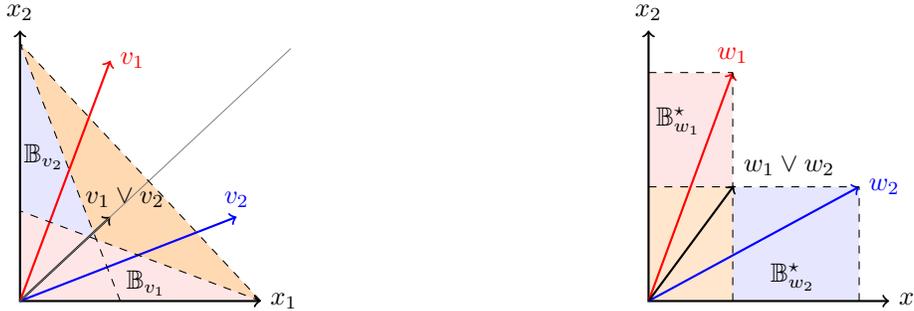
\begin{figure}[b!]
\begin{subfigure}{0.48\linewidth}
  \color{black}
  \centering
  \begin{tikzpicture}[scale=0.8]
% Figure on the right
    \fill[fill=orange!30 ] (0,0)-- (3.99999,0)--(0,4.2857);
    \fill[fill=blue!10 ] (0,0)-- (1.666666,0)--(0,4.2857);
\fill[fill=red!10 ] (0,0) --(3.99999,0)--(0,1.5);
    \draw [<->,thick] (0,4.5) node (yaxis) [above] {$x_2$}
        |- (4,0) node (xaxis) [right] {$x_1$};
         \draw[->, red,thick](0,0) -- (1.5,4) node [right]{$v_1$};
    \draw [->, blue, thick](0,0)--(3.6, 1.4) node [above] {$v_2$};
    \draw [->, black, thick](0,0)--(1.5, 1.4) node [ above] {$\;\;\;\;v_1\vee v_2$};
    \draw [-, gray, ultra thin](0,0)--(4.5, 4.2) node [right]{};
    \draw[dashed] (3.9999,0)--(0,1.5) ;
    \draw[dashed] (1.6666,0)--(0,4.2857) ;
        \draw[dashed] (3.9999,0)--(0,4.2857) ;
        \node at (0.4, 2.4)   (a) { $\mathbb{B}_{v_2}$};
\node at (2.1, 0.3)   (a) { $\mathbb{B}_{v_1}$};
  \end{tikzpicture}
  \caption{The unit balls $\B_{v_1}$ and $\B_{v_2}$ are represented in red and blue, respectively. In orange, the ball $\mathbb{B}_{v_1\vee v_2}=\B_{v_1}\ssquare \,\B_{v_2}$.} 
  \label{Fig:PrimalNormsBalls}
  \end{subfigure}
  \hskip0.3cm
   \begin{subfigure}{0.48\linewidth}
  \color{black}
  \centering
  \begin{tikzpicture}[scale=0.8]
% Figure on the right
    \fill[fill=red!10 ] (6,0) rectangle (7.4,3.8);
    \fill[fill=blue!10 ] (6,0) rectangle (9.5,1.9);
    \fill[fill=orange!20 ] (6,0) rectangle (7.4,1.9);
    \draw [<->,thick] (6,4.5) node (yaxis) [above] {$x_2$}
        |- (10,0) node (xaxis) [right] {$x_1$};
    \draw[->, red,thick](6,0) -- (7.4,3.8) node [above]{$w_1$};
    \draw [->, blue, thick](6,0)--(9.5, 1.9) node [right] {$w_2$};
    \draw [->, black, thick](6,0)--(7.42, 1.92) node [above right] {$w_1\vee w_2$};
\node at (8.4, 0.4)   (a) { $\mathbb{B}^\star_{w_2}$};
\node at (6.5, 3)   (a) { $\mathbb{B}^\star_{w_1}$};
    \draw[dashed] (7.4,0)--(7.4,3.8) ;
    \draw[dashed] (6,3.8)--(7.4,3.8) ;
      \draw[dashed] (6,1.9)--(9.5,1.9) ;
    \draw[dashed] (9.5,0)--(9.5,1.9) ;
\end{tikzpicture}
\caption{The unit balls $\B^\star_{w_1}$ and $\B^\star_{w_2}$ are represented in red and blue, respectively. In orange, we represented $\mathbb{B}^\star_{w_1\vee w_2}=\B^\star_{w_1}\cap \B^\star_{w_2}$.} 
  \label{Fig:DualNormsBalls}
  \end{subfigure}
  \caption{Examples of primal and dual copositive norms on $\R^2_{\geq 0}$, with the corresponding unit balls.}
  \label{Fig:PRimalDualnorms}
  \end{figure}
  
\begin{defn}\label{defn:CopNorms}
Given $v\in \R^n_{>0}$, we define the \emph{primal} and \emph{dual linear copositive norms induced by $v$} on $\R^n_{\geq 0}$ by
\begin{align}
g_v(x)&:=v^\top x,\,\,\,\, \,\,\,\,\,\,\text{and}\label{eq:PrimalNorms}\\
g_v^\star(x)&:=\max_i\left\{\frac{x_i}{v_i}\right\},\label{eq:DualNorms}
\end{align}
for all $ x\in \R^n_{\geq 0}$. We denote with $\cP$ and $\cD$ the set of all primal and dual copositive norms, respectively.
\end{defn}
One can show that the norms defined in expression~\eqref{eq:DualNorms} is exactly the dual of the one in~\eqref{eq:PrimalNorms}, in the sense of Definition~\ref{defn:DualNorms} in Appendix. See Figure~\ref{Fig:PRimalDualnorms} for a graphical interpretation of this class of functions.
In the context of positive switching systems, copositive norms as in~\eqref{eq:PrimalNorms} were considered, among many other examples, in~\cite{MasSho07,ForVal12}. In order to highlight the closure properties of these templates, we adopt the following notation; given $v,w\in \R^n_{>0}$, define $v\vee w\in \R^n_{>0}$ as 
\begin{equation}\label{eq:CompMin}
v \vee w~:=~\sum_i \min\{v_i,w_i\}\mathbf{e}_i,
\end{equation}
 i.e. the componentwise minimum between $v$ and $w$. 
In our preliminary paper~\cite{DDLJ21}, we focused on properties of primal copositive norms only. In what follows, we provide the statements and proofs for dual copositive norms, and moreover we show how, thanks to the convex-duality theory summarized in~\ref{Sec:Appendix}, the corresponding properties are re-obtained, as simple corollaries, for primal norms.
\begin{prop}[Properties of Dual Norms] \label{pro:DualProp}
 Given any $v,w\in  \R^n_{> 0}$, any $A\in \R^{n\times n}_{\geq 0}$, any $\lambda>0$ we have
 \begin{enumerate}[label=\emph(\arabic*$^d$)]
 \item $g_{v+w}^\star=g_v^\star\; \sharp\; g_w^\star$ (inverse summation\footnote{For the formal definition of the \emph{inverse summation} operation between convex functions (and the dual operation of \emph{infimal convolution}), see Properties~\ref{propertyis:Convexfubctuib} in~\ref{Sec:Appendix}} of dual norms is a dual norm) ;
 \item $g^\star_{\lambda v}=\frac{1}{\lambda}g_v^\star$; 
 \item $g^\star_{v\vee w}=\max\{g^\star_v, g^\star_w\}$ (max of dual norms is a dual norm); 
% %\item If $A$ is invertible, \|x\|_{Av}^\star=\|A^{-1}x\|_v^\star$.
 \item  $\left(\forall x\in\R^n_{\geq 0}, \,\,\,g^\star_v(Ax)\leq g^\star_w(x)\right)\,\,\,\Leftrightarrow \,\,\, Aw\leq_c v$. \label{Item:4dual}
 \end{enumerate}
 \end{prop}
\begin{proof}
Items \emph{(2$^d$)}, \emph{(3$^d$)} and \emph{(4$^d$)} are trivial from the definition in~\eqref{eq:DualNorms} (see also Figure~\ref{Fig:DualNormsBalls} for a graphical interpretation).  Given any $v\in \R^n_{>0}$, we adopt the notion $\B_v^\star:=\{x\in \R^n_{\geq 0}\;\;\vert \;\;g_v^\star(x)\leq 1\}$. For Item~\emph{(1$^d$)}, let us consider $v,w\in \R^n_{>0}$; using  Lemma~\ref{lemma:COrrespondenceBall} in Appendix, we have to show that $\B^\star_{v+w}=\B^\star_v+\B^\star_w$. \\
$(\supseteq)$: Consider $x\in \B^\star_v$ and $y\in \B^\star_w$, i.e. $x_i/v_i\leq 1$ and $y_i/w_i\leq 1$, for all $i \in \{1,\dots, n\}$. It follows that $x+y\in \B^\star_{v+w}$ since $x_i\leq v_i$ and $y_i\leq w_i$ implies $x_i+y_i\leq v_i+w_i$, for all $i\in \{1,\dots, n\}$.\\
$(\subseteq)$: Consider $z\in \B^\star_{v+w}$, i.e. $\frac{z_i}{v_i+w_i}\leq 1$, for all $i\in \{1,\dots, n\}$. Let us define $x,y\in \R^n_{\geq 0}$ by $x_i:=\frac{v_i}{v_i+w_i}z_i$ and $y_i:=\frac{w_i}{v_i+w_i}z_i$, for all $i\in \{1,\dots, n\}$. It is easy to see that $x+y=z$; moreover $x\in \B^\star_v$ and $y\in \B^\star_w$ since
\[
\frac{x_i}{v_i}=\frac{v_i}{v_i(v_i+w_i)}z_i=\frac{z_i}{v_i+w_i}\leq 1, \;\;\;\forall i\in \{1,\dots, n\},
\]
and similarly for $y$, concluding the proof.
\end{proof}
Item \emph{(1)$^d$} shows that the template $\cD$ is closed under the so-called inverse summation (see also \cite[Section 5]{RockConv}). In Section~\ref{sec:Lifts} we introduced the T-sum Lift (Definition~\ref{Def:SumLift}), which, intuitively, exploits and explores inequalities involving sums of the node-functions. We show next that, under the hypothesis of \emph{linear} sub-dynamics, the T-sum lift can provide information also on inverse summations of the node-functions, and it can thus be used for path-complete comparison purpose when the template is the set of dual copositive norms.
\begin{lem}\label{lem:MinUnderLinearMap}
Let $v_1,w_1, v_2, w_2\in \R^n_{>0}$ and $A\in\R^{n \times n}_{\geq 0}$. Then
\[ \left(\forall\,x\in \R^n_{\geq 0},\,\,\,g^\star_{v_2}(Ax)\leq g^\star_{v_1}(x)\;\;\wedge\; \;g^\star_{w_2}(Ax)\leq g^\star_{w_1}(x)\right) ~ \Rightarrow ~  \left(\forall \,x\in \R^n_{\geq 0},\,\,\,g^\star_{v_2+ w_2}(Ax)\leq g^\star_{v_1+ w_1}(x)\,\right).
\]
\end{lem}
\begin{proof}
Using Lemma~\ref{lemma:COrrespondenceBall} in~\ref{Sec:Appendix}, we have to show that
$
x\in  \B^\star_{v_1} +\B^\star_{w_1}\;\;\Rightarrow\;\;Ax\in  \B^\star_{v_2} +\B^\star_{w_2}$,
knowing that $x\in \B^\star_{v_1}\;\Rightarrow Ax\in \B^\star_{v_2}$ and $x\in \B^\star_{v_1}\;\Rightarrow Ax\in \B^\star_{v_2}$. Consider $x\in  \B^\star_{v_1} +\B^\star_{w_1}$, i.e. $x=y+z$, with $y\in \B^\star_{v_1}$ and $z\in \B^\star_{w_1}$. Computing, we have $Ax=A(y+z)=Ay+Az\in \B^\star_{v_2}+\B^\star_{w_2}$, concluding the proof.
\end{proof}
Now that we have studied the analytical properties of the template of copositive dual norms $\cD$, we are able to state the corresponding consequences on the validity of the template dependent lifts presented in Section~\ref{sec:Lifts}.
 \begin{thm}[Valid Lifts for Dual Copositive Norms]\label{thm:ValidLiftDUalsNorms}
 Consider $\cG=(S,E)$ a path-complete graph on $\M$, and any $T\in \N$. Denote  by $\cL$ the set of all the $M$-tuples of nonnegative matrices. We have
 \begin{itemize}
\item $\cG~\leq_{\cD,\cL}\,\cG^{\oplus T}$,
\item $\cG~\leq_{\cD,\cL} \,\cG_{\text{max}}$,
    \end{itemize}
\end{thm}
\begin{proof}
The proof follows from  Theorem~\ref{Thm:GeneralThmLifts} and by Proposition~\ref{pro:DualProp} and Lemma~\ref{lem:MinUnderLinearMap}.
\end{proof}

Now, thanks to the convex duality theory summarized in Appendix, we can straightforwardly obtain similar results for the template $\cP$ of \emph{primal} copositive norms. The same statements (but with direct proofs) can be found in our preliminary paper~\cite{DDLJ21}.

\begin{prop}[Properties of Primal Norms]\label{prop:Primal}
Given any $v,w\in  \R^n_{> 0}$, any $A\in \R^{n\times n}_{> 0}$, any $x\in \R^n_{\geq 0}$, any $\lambda>0$ we have
\begin{enumerate}[label=\emph(\arabic*$^p$)]
\item $g_{v+w}=g_v+g_w$ (sum of dual norms is a dual norm);
\item $g_{\lambda v}=\lambda g_v$;
\item $g_{v \vee w}=g_v\ssquare g_w$ \label{Item:Infimalconvolution}(infimal convolution of dual norms is a dual norm);
%\item If $A$ is invertible $\|Ax\|_v=\|x\|_{A^\top v}$
\item $\left(\forall x\in\R^n_{\geq 0}, \,\,g_v(Ax)\leq g_w(x)\right)\,\,\,\Leftrightarrow \,\,\, A^\top v\leq_c w$. \label{item:4Primal}
\end{enumerate}
\end{prop}
For the formal definition of the infimal convolution operation, see Properties~\ref{propertyis:Convexfubctuib} in~\ref{Sec:Appendix}. The proof follows, by duality, using Proposition~\ref{pro:DualProp} and Lemma~\ref{lem:DualNormsAppendix} in~\ref{Sec:Appendix}.
\begin{thm}[Valid Lifts for Primal Copositive Norms]\label{thm:LiftforPrimal}
 Consider $\cG=(S,E)$ a path-complete graph on $\M$ and any $T\in \N$. Denote by $\cL$ the set of all the $M$-tuples of nonnegative matrices. We have
\begin{itemize}
\item $\cG~\leq_{\cP,\cL}\,\, \,\,\,\cG^{\oplus T}$,
\item $\cG~\leq_{\cP,\cL} \,\cG_{\text{min}}$,
    \end{itemize}
\end{thm}
This is obtained, by duality, from Theorem~\ref{thm:ValidLiftDUalsNorms}.
 
Concluding this section, we provide a general duality result for the path-complete criteria comparison problem for primal and dual copositive norms.
\begin{cor}
 Given any $\cG_1,\cG_2$ path-complete graphs on $\M$, we have that:
 \begin{equation}\label{eq:DualityforPrimalDualCopositive}
 \cG_1~\leq_{\cP,\cL}~\cG_2\;\;\;\Leftrightarrow\;\;\;\cG_1^\top~\leq_{\cD,\cL}~\cG_2^\top.
\end{equation}
 \end{cor}
 \begin{proof}
It suffices to recall that, from Items \emph{(4$^d$)},\emph{(4$^p)$} of Lemmas~\ref{pro:DualProp} and~\ref{prop:Primal}, given any $v,w\in  \R^n_{> 0}$ and any $A\in \R^{n\times n}_{\geq 0}$ the following equivalence holds:
\begin{equation*}
\forall\,x\in \R^n_{\geq 0}\,\,\,g_v(Ax)\leq g_w(x)\,\,\Leftrightarrow\,\,\forall\,x\in \R^n_{\geq 0}\,\,\, g^\star_w(A^\top x)\leq g^\star_v(x).
\end{equation*}
Indeed, this equivalence implies~\eqref{eq:DualityforPrimalDualCopositive}, by the general result presented in Lemma~\ref{lemma:DualInequallity} in Appendix.
 \end{proof}

\subsection{Approximation of JSR with arbitrary accuracy}

In this section we show how the templates $\cP$ and $\cD$ studied in Section~\ref{subsec:CopNorms} can provide an estimation with arbitrary accuracy of the \emph{joint spectral radius} of a set of nonnegative matrices, using a path-complete Lyapunov approach. While the same kind of theoretic estimation was already provided in~\cite[Section 6]{AJPR:14} for generic matrices and quadratic Lyapunov functions, our result will lead to a new hierarchy of \emph{linear programs} (instead of semidefinite programs), thus drastically reducing the computation complexity. We recall how previous hierarchies of LPs (approximating the JSR of nonnegative matrices), as the one proposed in~\cite[Corollary 3]{ProJunBlo10}, require, in general, the computation of long products of matrices in the considered set. Our approach, instead, is not affected by this drawback, which is avoided requiring a more complex structure of the candidate path-complete Lyapunov function, as we will develop in what follows. Moreover, we show how the results concerning \emph{lifts} presented in Section~\ref{sec:Lifts} specialize in this setting, leading to ``smart'' choices of graph structures for the stability analysis.

In estimating the JSR of a set of non-negative matrices $\cA=\{A_1,\dots A_M\}\subset \R^{n\times n}_{\geq 0}$, the first possible approach is to consider \emph{common} Lyapunov function in the templates $\cP$ and $\cD$. We thus denote by $\rho_{\cP}(\cA)$, the quantity
\[
\begin{aligned}
\rho_{\cP}(\cA):=\;\;\;\;\;\;&\inf_{v\in \R^n_{>0},\gamma\geq 0} \gamma\\
&A_i^\top v-\gamma v\leq 0, \;\;\;\;\forall i\in \M.
\end{aligned}
\]
see also \cite{ProJunBlo10}. Intuitively, $\rho_{\cP}(\cA)$ represents the best estimate of the joint spectral radius of $\cA$ one can obtain considering \emph{common copositive primal Lyapunov norms}, recall Item~\ref{item:4Primal} of Proposition~\ref{prop:Primal}.
Similarly we can define $\rho_{\cD}(\cA)$, by
\[
\begin{aligned}
\rho_{\cD}(\cA):=\;\;\;\;\;\;&\inf_{v\in \R^n_{>0}, \gamma\geq 0} \gamma\\
&A_i v-\gamma v\leq 0, \;\;\;\;\forall i\in \M.
\end{aligned}
\]
Again, $\rho_{\cD}(\cA)$ represents the best bound on the JSR of $\cA$ considering common dual copositive Lyapunov norms, recall Item~\ref{Item:4dual} of Proposition~\ref{pro:DualProp}. We recall the following useful properties:
\begin{thm}[\cite{ProJunBlo10}]\label{thm:ConicRadiuss}
Consider $\cA:=\{A_1,\dots ,A_M\}\subset \R^{n \times n}_{\geq 0}$. We have
\begin{enumerate}
    \itemsep0em 
\item $\frac{1}{n}\rho_{\cP}(\cA)\leq \rho(\cA)\leq \rho_{\cP}(\cA)$,
\item $\frac{1}{n}\rho_{\cD}(\cA)\leq \rho(\cA)\leq \rho_{\cD}(\cA)$,
\item $\rho_{\cP}(\cA)=  \rho_{\cD}(\cA^\top)$.
	\end{enumerate}
\end{thm}
\begin{proof}
For Item \emph{1.} see~\cite[Theorem 2.6]{ProJunBlo10}. Item \emph{3.} follows by definition, and finally Item \emph{2.} follows by Items \emph{1.} and \emph{3.} recalling that $\rho(\cA)=\rho(\cA^\top)$.
\end{proof}

By Theorem~\ref{thm:ConicRadiuss}, the approximation guarantees provided by $\rho_{\cP}(\cA)$ and $\rho_{\cD}(\cA)$ are the same. On the other hand, for \emph{a specific set} $\cA\subset \R^{n\times n}_{\geq 0}$, it is possible that one template (primal/dual copositive norms) will lead to a better estimation of the JSR. The following simple example shows how, for a particular set of matrices $\cA$, the choice of primal or dual copositive norms is crucial in estimating the joint spectral radius, and in particular the inequalities in Items \emph{1.} and \emph{2.} of~Theorem~\ref{thm:ConicRadiuss} can be tight.

\begin{exmp}\label{ex:PrimalVSDual}
Fix a dimension $n\in \N$ and consider $\cA:=\{A_1,\dots ,A_n\}\subset \R^{n\times n}_{\geq 0}$ defined by $A_i=\mathbf{1}\, \mathbf{e}_i^\top $ for $i\in \langle n \rangle$, i.e.
\[
A_1 = \begin{bmatrix} 
    1 & \dots & 0 \\
    \vdots  & \dots&\vdots \\
    1     & \dots &0 
    \end{bmatrix},\;\;\dots,\;\; A_n = \begin{bmatrix} 
     0 & \dots & 1 \\
     \vdots& \dots &\vdots \\
    0& \dots &1 
    \end{bmatrix},
\]
where $\mathbf{1}:=(1\dots,1)^\top$.
From straightforward computation it holds that $\rho(\cA)=1$. Computing, we have, for all $x\in \R^n_{\geq 0}$, $A_ix=x_i\mathbf{1}$, and thus considering the norm $g^\star_{\mathbf{1}}$ (the usual infinity norm) we have that, for all $i\in \langle n\rangle$ and for any $x\in \R^n_{\geq 0}$ $g^\star_{\mathbf{1}}(A_ix)\leq g^\star_{\mathbf{1}}(x)$, proving that $\rho_\cD(\cA)=\rho(\cA)$. In other words, dual copositive common Lyapunov norms provide an exact estimation of the JSR. It can be seen that $\rho_\cP(\cA)=n$, that is, by Theorem~\ref{thm:ConicRadiuss}, the worst possible estimate. The ``dual case'', i.e. considering $\cA^\top$, provides an example for which the primal norms provides an exact estimate, and the dual ones the worst possible. \hfill $\triangle$
\end{exmp} 

We now show that, considering multiple copositive primal/dual norms (or, more precisely, path-complete Lyapunov functions in these templates), we can provide an estimation of the JSR with arbitrary accuracy. Given $\cA:=\{A_1,\dots, A_M\}\subset \R^{n\times n}$ and a path-complete graph $\cG=(S,E)$ the quantity $\rho_{\cP,\cG}(\cA)$ and $\rho_{\cD,\cG}(\cA)$ are defined as the optimal values of the problems 

\begin{equation}\label{eq:LPprimalNorms}
\begin{aligned}
\rho_{\cP,\cG}(\cA):=\;\;\;\;\;\;&\inf_{v_1,\dots v_{|S|}\in \R^n_{>0},\gamma\geq 0} \gamma\\
&A_i^\top v_b-\gamma v_a\leq 0, \;\;\;\;\forall\, e=(a,b,i)\in E.
\end{aligned}
\end{equation}
and 
\begin{equation}\label{eq:LPDualNorms}
\begin{aligned}
\rho_{\cD,\cG}(\cA):=\;\;\;\;\;\;&\inf_{v_1,\dots v_{|S|}\in \R^n_{>0},\gamma\geq 0} \gamma\\
&A_iv_b-\gamma v_a\leq 0, \;\;\;\;\forall\, e=(a,b,i)\in E.
\end{aligned}
\end{equation}
We now define an important class of graphs that will be used in our JSR approximation scheme.
\begin{defn}[De Bruijn Graphs]\label{defn:DeBRunjii}
Given $M,l\in \N$  the \emph{(primal) De Bruijn graph of order $l-1$ (on the alphabet $\M$)} denoted by $\cG^l_{\emph{db}}=(S,E)$ is defined as follows: $S:=\M^{l-1}$ and, given any node $a=(i_1, \dots,i_{l-1})\in S$, we have $(a,b, j)\in E$ for every $b$ of the form $b=(i_2,\dots, i_{l-1},j)$, for any $j\in \M$.
\end{defn}
We note that $\cG^l_{\emph{db}}$ is complete (recall Definition~\ref{def:ComplateGraph}). The \emph{dual De Bruijn graph of order $l-1$ (on the alphabet $\M$)}, denoted by $(\cG^l_{\emph{db}})^\top$, is thus co-complete. For further discussion on the class of De Bruijn graphs see~\cite[Section 6]{AJPR:14}.

\begin{thm}[``Asymptotic'' Converse Lyapunov Theorem]\label{thm:LyaounovConverseTheorem}
Let $\cA=\{A_1,\dots, A_M\}\subset \R^{n\times n}_{\geq 0}$. Given any $l\in \N$, considering $\cG^l_{\emph{db}}=(S,E)$ the (primal) De Bruijn graph of order $l-1$ on $\M$, we have 
\begin{equation}\label{eq:asymtoticBounds}
\frac{1}{\sqrt[l]{n}}\rho_{\cD,\cG^l_{\emph{db}}}(\cA)\leq \rho(\cA)\leq \rho_{\cD,\cG^l_{\emph{db}}}(\cA).
\end{equation}
\end{thm} 

\begin{proof}
 Recalling Definition~\ref{defn:DeBRunjii} and Proposition~\ref{pro:DualProp}, $\cG^l_{\emph{db}}=(S,E)$ leads to the inequalities:
\begin{equation}\label{eq:InequalitiesGraph}
\begin{aligned}
v_{(i_1,\dots,i_{l-1})}&>0,\;\; \;\forall (i_1,\dots, i_{l-1})\in \M^{l-1}\\
A_j v_{(i_1, \dots, i_{l-1})}&\leq \gamma v_{(i_2,\dots,i_{l-1},j)},\;\; \;\forall (i_1,\dots, i_{l-1})\in \M^{l-1},\;\forall \,j\in \M.
\end{aligned}
\end{equation}
We now prove that $\rho_{\cP,\cG^l_{\emph{db}}}(\cA)$, that is the minimum $\gamma$ for which \eqref{eq:InequalitiesGraph} is feasible, satisfies the inequalities in~\eqref{eq:asymtoticBounds}.
The inequality $\rho(\cA)\leq \rho_{\cP,\cG^l_{\emph{db}}}(\cA)$ is straightforward. Consider now $\cA^l$, the set of all the possible products of matrices in $\cA$ of length $l$. By Theorem~\ref{thm:ConicRadiuss}, we have $\frac{1}{n}\rho_{\cP}(\cA^l)\leq \rho(\cA^l)=\rho(\cA)^l$ and thus
\begin{equation}\label{eq:InequalityProof}
\frac{1}{\sqrt[l]{n}}\sqrt[l]{\rho_{\cP}(\cA^l)}\leq \rho(\cA).
\end{equation}
We suppose thus that $v\in \R^n_{>0}$ is such that $g^\star_v$ is a common copositive dual Lyapunov norm for $\cA^l$ with decay $\gamma^l>0$, i.e.
\[
\begin{aligned}
v&>0,\\
A_{i_l}\cdots A_{i_1} v-\gamma^l v&\leq 0,\;\;\;\forall (i_1,\dots,i_l)\in \M^l.
\end{aligned}
\]
For any $(i_1,\dots, i_{l-1})\in \M^{l-1}$, defining
\[
v_{(i_1,\dots, i_{l-1})}:=v+\frac{1}{\gamma}A_{l-1} v+\frac{1}{\gamma^2}A_{i_{l-1}}A_{i_{l-2}} v+\dots +\frac{1}{\gamma^{l-1}}A_{i_{l-1}}\cdots A_{i_1}  v,
\]
it is easy to see that inequalities in \eqref{eq:InequalitiesGraph} are satisfied. We have thus proved that $\rho_{\cD,\cG^l_{\emph{db}}}(\cA)\leq \sqrt[l]{\rho_{\cD}(\cA^l)}$, and recalling~\eqref{eq:InequalityProof} we conclude.
\end{proof}
A similar proof is used in~\cite[Theorem 6.2]{AJPR:14} for the template of quadratic functions, obtaining an approximation guarantee for a hierarchy of semidefinite programs.
\begin{rem}\label{Rem:DualityDeBrunjii}
We note that, since the graph $\cG^l_{\emph{db}}$ is complete, recalling Proposition~\ref{prop:CompletenessandMAxMin}, if the inequalities in~\eqref{eq:InequalitiesGraph} are feasible (for a certain $\gamma>0$), we also have that the function defined by
\[
V_{\cG^l_{\emph{db}}}(x):=\min_{(i_1,\dots, i_{l-1})\in \M^{l-1}}\left\{g^\star_{v_{(i_1,\dots, i_{l-1})}}(x)\right\}
\]
is a common Lyapunov function for~\eqref{eq:PosSystem}. 
From Theorem~\ref{thm:LyaounovConverseTheorem} we can obtain its dual result: applying again the duality relation in~\eqref{eq:DualityforPrimalDualCopositive} we have that, for any $l\in \N$,
\[
\frac{1}{\sqrt[l]{n}}\rho_{\cP,(\cG^l_{\emph{db}})^\top}(\cA)\leq \rho(\cA)\leq \rho_{\cP,(\cG^l_{\emph{db}})^\top}(\cA),
\]
where $(\cG^l_{\emph{db}})^\top$ denote the dual De Bruijn graph of order $l-1$. Moreover, the conditions encoded in $(\cG^l_{\emph{db}})^\top$ will define, again by duality, a common Lyapunov function for~\eqref{eq:PosSystem}, in the form of a \emph{max} of primal copositive norms. We note that (convex hull of) min of dual copositive norms and max of primal copositive norms are special cases of \emph{polyhedral functions}. In this view, Theorem~\ref{thm:LyaounovConverseTheorem} states in particular that, if the system~\eqref{eq:PosSystem} is asymptotically stable, then there exists a copositive polyhedral common Lyapunov function. This is consistent  with (and strenghtens) the universality of polyhedral Lyapunov functions for switching systems proved in~\cite{BM1999}, see also~\cite{AJ2020}. \hfill $\triangle$
\end{rem}
\begin{rem}[Numerical approximation of the JSR via De Bruijn Hierarchy]\label{rem:DeBrunjii}
Given $\cA=\{A_1,\dots,A_M\}\subset\R^{n\times n}_{\geq 0}$, Theorem~\ref{thm:LyaounovConverseTheorem} (and the subsequent Remark~\ref{Rem:DualityDeBrunjii})
suggests the following numerical scheme in order to approximate $\rho(\cA)$ with arbitrary precision, using the hierarchies of primal and dual De Bruijn graphs. This scheme is summarized in the following pseudo-algorithm.
\begin{itemize}[leftmargin=2cm]
\item [(\emph{Init.}):]Fix a margin $\varepsilon>0$, set $l=1$, $\underline \gamma=0$, $\overline \gamma=+\infty$.
\item[] {\bf since} $(\overline \gamma-\underline \gamma\geq \varepsilon)$, 
\item[(\emph{Step $l$}):]  Solve the linear program in~\eqref{eq:LPDualNorms} for $\cG^l_{db}$. \\Set $\underline \gamma\leftarrow\max\{\underline{\gamma},\;\;\frac{1}{\sqrt[l]{n}}\rho_{\cD,\cG^l_{\emph{db}}}(\cA)\}$  and \\$\overline{\gamma}\leftarrow\min\{\overline{\gamma},\;\;\rho_{\cD,\cG^l_{\emph{db}}}(\cA)\}$.
\item[(\emph{Step $l^d$}):] Solve the linear program in~\eqref{eq:LPprimalNorms} for $(\cG^l_{db})^\top$.\\ 
Set $\underline \gamma\leftarrow\max\{\underline{\gamma},\;\;\frac{1}{\sqrt[l]{n}}\rho_{\cP,(\cG^l_{\emph{db}})^\top}(\cA)\}$  and \\$\overline{\gamma}\leftarrow\min\{\overline{\gamma},\;\;\rho_{\cP,(\cG^l_{\emph{db}})^\top}(\cA)\}$.
\item[] $l\leftarrow l+1$.
\end{itemize} 
This procedure allows us, once a confidence margin $\varepsilon>0$ is chosen, to provide tight estimations  of JSR of nonnegative matrices. Other stopping criteria can be considered, as for example the condition $(\overline \gamma <1)$ which ensures asymptotic stability of~\eqref{eq:PosSystem}, or $(\underline \gamma>1)$ which is an instability certificate for~\eqref{eq:PosSystem}. \hfill $\triangle$
\end{rem}
\subsection{Numerical example}\label{subsection:NumericalExample}

In this last section, we consider a positive switching system as in~\eqref{eq:PosSystem} and we provide the following analysis: first, given a particular path-complete graph $\cG$, we see how, when considering dual copositive norms, the estimation of the JSR is improved considering the max-lift $\cG_{max}$, in line with Theorem~\ref{thm:ValidLiftDUalsNorms}. Secondly, applying the idea of Remark~\ref{rem:DeBrunjii}, we provide an accurate estimation of the joint spectral radius. 
\begin{exmp}\label{ex:NumericalExample}
We consider the positive switching system~\eqref{eq:PosSystem} defined by $\cA=\{A_1,A_2\}\subset\R^{3 \times 3}_{\geq 0}$ with

\begin{equation}\label{Eq:ExNumerique}
A_1=  \begin{bmatrix} 0.2 & 0 & 0\\ 0.6 & 0.6 & 0.5\\ 0.6 & 0.3 & 0.2 \end{bmatrix}, \,\,\,A_2= \begin{bmatrix} 0.1 & 0.2 & 0.3\\0.2 & 0 & 0.5\\ 0.1  & 0.6 & 0.7 \end{bmatrix}.
\end{equation}
First, in order to approximate $\rho(\cA)$, we solve the problem in~\eqref{eq:LPDualNorms} for $\cG_5$ in Figure~\ref{Fig:InitialGraphExNum}, obtaining $\rho_{\cG_5,\cD}(\cA)=1.3075$. We consider the max lift  $(\cG_5)_{\max}$ and in particular we select a path-complete and strongly connected component of $(\cG_5)_{\max}$ given by $\cG_6$ in Figure~\ref{Fig:LiftedGraphExNum}. We know, by Theorem~\ref{thm:ValidLiftDUalsNorms} and Proposition~\ref{Prop:InegalityHoldsForSubgraph} that $\cG_5\leq_{\cD,\cL} \cG_6$ and thus we expect that $\rho_{\cG_6,\cD}(\cA)\leq \rho_{\cG_5,\cD}(\cA)$ which is confirmed, since solving~\eqref{eq:LPDualNorms}, we obtain $\rho_{\cG_6,\cD}(\cA)=1.2716$. It is interesting to note how the graph $\cG_6$, although it reduces the number of decision variables and inequalities with respect to the conditions encoded in $\cG_5$, provides a better estimation of the JSR. Actually, given a positive system, we know that $\cG_6$ will provide at worst the same estimation as $\cG_5$ and for some particular cases as (\ref{Eq:ExNumerique}), $\cG_6$ will provide a strictly better approximation than $\cG_5$. This highlights that, given a particular path-complete structure and a template, the lifting approach can provide better estimation of the joint spectral radius while decreasing the number of Lyapunov inequalities and decision variables.

Concluding, we provide upper and lower bounds for $\rho(\cA)$ using the idea of Remark~\ref{rem:DeBrunjii}. For simplicity, we stop at the fourth iteration of the numerical scheme (and thus considering until the primal and dual De Bruijn graphs of order $3$) obtaining the following results:
\renewcommand{\arraystretch}{1.3}
\begin{center}
\begin{tabular}{ |c|c|c|c|c|c|c|c|c|} 
 \hline
 Steps: & $ (1)$ & $(1)^d$ & $(2)$ & $(2)^d$ & $(3)$ & $(3)^d$& $(4)$ & $(4)^d$ \\ 
 \hline
 $\rho_\cG$ & $1.445$ & $1.341$ & $1.445$ & $1.070$ & $1.410$ & $1.070$ & $1.402$ & $1.070$\\
 \hline 
$\underline \gamma$ & $0.482$ & $0.482$ & $0.834$ & $0.834$ & $0.978$ & $0.978$ & $1.065$ & $1.065$ \\ 
 \hline
 $\overline \gamma$ & $1.445$ & $1.341$ & $1.341$ & $1.070$ & $1.070$ & $1.070$ & $1.070$ & $1.070$  \\  
 \hline

\end{tabular}
\end{center}
In this table, in the line denoted by $\rho_\cG$ we reported the optimal values of the LPs described by~\eqref{eq:LPprimalNorms},~\eqref{eq:LPDualNorms} for the corresponding (primal and dual) De Bruijn graphs. We have thus  proven that $\rho(\cA)\in [1.065, 1.070]$, having an instability certificate for the positive switching system~\eqref{eq:PosSystem} defined by $\cA$. It is interesting to note how, in this particular case, the conditions arising from the primal De Bruijn graphs and the template of dual copositive norms provide better upper bounds for the JSR. 
\hfill $\triangle$
\end{exmp}

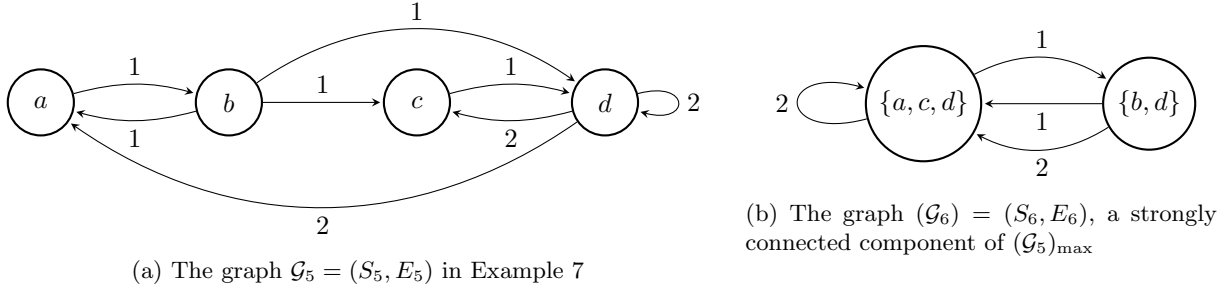
\begin{figure}[t!]
\begin{subfigure}{0.61\linewidth}
  \color{black}
  \centering
  \begin{tikzpicture}%
  [>=stealth,
  shorten >=1pt,
  node distance=1cm,
  on grid,
  auto,
  every state/.style={draw=black, fill=white,thick}
  ]
  \node[state] (node1)                 {$a$};
  \node[state] (node2) [right=of node1, xshift=1.5cm] {$b$};
  \node[state] (node3) [right =of node2, xshift=1.5cm]{$c$};
  \node[state] (node4) [right =of node3, xshift=1.5cm]{$d$};
  \path[->]
  %   FROM       BEND/LOOP           POSITION OF LABEL   LABEL   TO
  (node4) edge[loop right=60]     node                      {2} (node4)
  (node1) edge[bend left=15]     node                      {1} (node2)
  (node2) edge[bend left=15]     node                      {1} (node1)
  (node2) edge[bend left=0]     node                      {1} (node3)
    (node2) edge[bend left=35]     node                      {1} (node4)
    (node3) edge[bend left=15]     node                      {1} (node4)
    (node4) edge[bend left=15]     node                      {2} (node3)
    (node4) edge[bend left=35]     node                      {2} (node1)
;
  \end{tikzpicture}
  \caption{The graph $\cG_5 = (S_5,E_5)$ in Example~\ref{ex:NumericalExample}}
  \label{Fig:InitialGraphExNum}
  \end{subfigure}
  \begin{subfigure}{0.38\linewidth}
  \color{black}
  \centering
  \begin{tikzpicture}%
  [>=stealth,
  shorten >=1pt,
  node distance=1cm,
  on grid,
  auto,
  every state/.style={draw=black, fill=white,thick}
  ]
  \node[state] (node1)         {$\{a,c,d\}$};
  \node[state] (node2) [right=of node1, xshift=2cm] {$\{b,d\}$};
  \path[->]
  %   FROM       BEND/LOOP           POSITION OF LABEL   LABEL   TO
  (node1) edge[loop left=180]     node                      {2} (node1)
  (node1) edge[bend left=30]     node                      {1} (node2)
  (node2) edge[bend left=0]     node                      {1} (node1)
  (node2) edge[bend left=30]     node                      {2} (node1)
;
  \end{tikzpicture}
  \caption{The graph $(\cG_6)=(S_6,E_6)$, a strongly connected component of $(\cG_5)_{\max}$}
\label{Fig:LiftedGraphExNum}
  \end{subfigure}
  \caption{The graphs $\cG_5$ and $\cG_6$ in Example~\ref{ex:NumericalExample}}
\end{figure}

\section{Conclusion}

Path-complete stability criteria for switching systems are based on two main components: the path-complete graph, which encodes the required Lyapunov inequalities, and the template, the set in which the candidate Lyapunov functions are selected.
In this article we provided new results concerning the comparison problem for path-complete Lyapunov conditions. We introduced new formal transformations of path-complete graphs, called lifts, which allow us to establish order relations between graphs. We analyzed how the effectiveness of these lifts strongly depends on the closure properties of the chosen template. This allowed us to generalize previous results and to provide a unifying framework which enables for finer comparison criteria between path-complete techniques. As particular case study, we thoroughly analyzed the template of primal and dual copositive norms, which provided a handy but yet intriguing framework in order to provide new stability results, with applications to the stability analysis of positive switching systems. The results have been validated with several examples. For the future, we plan to proceed further in this analysis, with particular attention to the template of quadratic functions, which is probably the most common (for both theoretical and numerical reasons) candidate Lyapunov functions template in control theory.

\section*{Acknowledgements}
Raphaël M. Jungers is a FNRS honorary Research Associate. This project has received funding from the European Research Council (ERC) under the European Union's Horizon 2020 research and innovation programme under grant agreement No 864017 - L2C. RJ is also supported by the Walloon Region, the Innoviris Foundation, and the FNRS (Chist-Era Druid-net).

\appendix
\section{Convex Duality: Gauge Functions, Polar Sets, and Dual Norms}\label{Sec:Appendix}

In this short Appendix we provide a complete summary of the convex duality results used in highlighting duality relations for the problem of comparison of stability criteria induced by graphs. These classical statements are in particular beneficial when studying the particular case of (primal and dual) copositive linear norms as template in studying stability of positive switching systems, which is the content of Section~\ref{sec:Positive}. The notation and terminology of this summary are introduced in~\cite[Part III]{RockConv}, in which the interested reader can find the corresponding formal proofs. For notational simplicity we develop the theory on $\R^n$; the corresponding statements for the self-dual cone $\R^n_{\geq 0}$ (as in Section~\ref{sec:Positive}) are straightforwardly obtained, mutatis mutandis.
 We consider subsets of $\R^n$ satisfying the following properties.
\begin{defn}\label{ass:MainSets}
Given $n\in \N$, $\cK(\R^n)$ denotes the family of sets $C\subset \R^n$ such that $C$ is closed, bounded, convex, symmetric ($x\in C$ if and only if $-x\in C$) and $0\in \inn(C)$. \hfill $\triangle$
\end{defn}
\begin{proper}
Consider $C_1,C_2\in \cK(\R^n)$, $A\in \R^{n\times n}$ invertible. Then:
\begin{enumerate}
    \itemsep0em 
\item \emph{(Sum):} $C_1+C_2\in \cK(\R^n)$,
\item\emph{(Intersection):} $C_1\cap C_2\in \cK(\R^n)$,
\item \emph{(Convex Hull of Union):} $C_1 \ssquare C_2 :=\co\{C_1\cup C_2\}\in \cK(\R^n)$,
\item \emph{(Inverse Sum):} $C_1 \sharp C_2:=\bigcup_{\lambda\in [0,1]}\lambda C_1\cap (1-\lambda)C_2\;\in \cK(\R^n)$,
\item \emph{(Linear Transform):} $AC_1\in \cK(\R^n)$.
\end{enumerate}
\end{proper}
\begin{defn}[Polar Sets]\label{Defn:Polar}
Given $C\subset \R^n$ convex, closed and such that $0\in C$, we define the \emph{polar} of $C$, denoted by $C^\circ$, by
\[
C^\circ:=\{x\in \R^n\;\vert\;\sup_{y\in C}\inp{y}{x}\leq 1 \}.
\] 
It can be seen that $C^\circ$ is closed, convex and $0\in C^\circ$ and moreover, $(C^\circ)^\circ=C$.\hfill $\triangle$
\end{defn}
It is easy to see that $C\in \cK(\R^n)$ if and only if $C^\circ\in \cK(\R^n)$. Moreover we have the following relations:
\begin{lem}\label{lemma:SetsProperties}
Consider $C_1,C_2\in \cK(\R^n)$, $A\in \R^{n\times n}$. Then
\begin{enumerate}
    \itemsep0em 
    \item $C_1\subset C_2 \,\,\Leftrightarrow \,\, C_2^\circ \subset C_1^\circ$,
    \item For every $\gamma>0$, $(\gamma C_1)^\circ=\frac{1}{\gamma}C_1^\circ$,
\item $(C_1+C_2)^\circ=C_1^\circ\sharp C_2^\circ$,
\item $(C_1\sharp C_2)^\circ=C_1^\circ +C_2^\circ$,
\item $(C_1\cap C_2)^\circ=C_1^\circ \ssquare C_2^\circ$,
\item $(C_1\ssquare C_2)^\circ=C_1^\circ \cap C_2^\circ$,
\item $(AC_1)^\circ=A^{-\top}\,C_1^\circ$.
\end{enumerate}
\end{lem}
There is a $1$-to-$1$ correspondence between sets in $\cK(\R^n)$ and norms on $\R^n$. This correspondence is induced  by the unit-sublevel sets of norms and by the Gauge functions of sets (respectively) in $\cK(\R^n)$, as recalled in what follows.

\begin{defn}
Given a set $C\in \cK(\R^n)$ we define the \emph{Gauge function associated to $C$},  $g(\,\cdot\,\vert \, C):\R^n\to \R$ by
\[
g(\,x\,\vert\,C):=\inf\{\gamma\in \R\;\vert\; x\in \gamma C,\;\gamma\geq 0\}.
\]
\end{defn}
Let us consider 
$
\cV(\R^n):=\{f:\R^n\to \R\;\vert\; \text{$f$ is a norm}\},
$
and, given $f\in \cV(\R^n)$ denote with
$
\B_f:=\{x\in \R^n\;\vert\;f(x)\leq 1\},
$
the \emph{unit ball} of $f$.
\begin{lem}
If $f\in \cV(\R^n)$ then $\B_f\in \cK(\R^n)$, and moreover
\[
f(x)=g(x\,\vert\, \B_f),\;\;\forall \,x\in \R^n.
\]
Conversely, for any $C\in \cK(\R^n)$, $g(\cdot\,\vert\,C)\in \cV(\R^n)$. More esplicilty,
$
\cV(\R^n)=\left\{g(\,\cdot\,\vert\,C):\R^n\to \R\,\,\vert\,\,C\in \cK(\R^n)\right\}$ and $
\cK(\R^n)=\left\{\B_f\subset \R^n\,\,\vert\,\,f\in \cV(\R^n)\right\}$.
\end{lem}

\begin{defn}[Dual Norm]\label{defn:DualNorms}
Given $f\in \cV(\R^n)$, we denote the \emph{dual norm of $f$}, denoted by $f^\star:\R^n\to \R$ by
\[
f^\star(x):=\sup_{y\in \R^n\setminus \{0\}}\frac{\inp{y}{x}}{f(x)}=\sup_{y\in \R^n,\,f(y)=1}\inp{y}{x}.
\]
It can be seen that  $f^\star=g(\cdot\,\vert\, \B_f^\circ)$ (and thus $\B_{f^\star}=\B_f^\circ$) and $(f^\star)^\star=f$. \hfill $\triangle$
\end{defn}

\begin{proper}\label{propertyis:Convexfubctuib}
Given $f_1,f_2\in \cV(\R^n)$, and $A\in \R^{n\times n}$ invertible, we have
\begin{enumerate}
    \itemsep0em 
\item\emph{(Sum):} $f_1+f_2\in \cV(\R^n)$,
\item \emph{(Max):} $\max\{f_1,f_2\}\in \cV(\R^n)$,
\item \emph{(Infimal Convolution):} Defining $f_1\ssquare f_2:\R^n\to \R$ by
\[
f_1\ssquare f _2(x):=\inf_{x=x_1+x_2}\{f(x_1)+f(x_2)\} ,
\]
it holds that $f_1\ssquare f_2\in \cV(\R^n)$. (Note that $f_1\ssquare f_2=\co\{\min\{f_1,f_2\}\}$, where $\co(f)$ denote the largest convex function majorized by $f$),
\item \emph{(Inverse Summation):} Defining $f_1\sharp f_2:\R^n\to \R$ by
\[
f_1\sharp f_2(x):=\inf_{x=x_1+x_2}\{\max\{f_1(x_1),f_2(x_2)\},
\]
we have $f_1\sharp f_2\in \cV(\R^n)$,
\item \emph{(Linear Transform):} $f_1\circ A\in \cV(\R^n)$.
\end{enumerate}
\end{proper}
\begin{lem}[Correspondence with Unit Balls]\label{lemma:COrrespondenceBall}
Given $f_1,f_2\in \cV(\R^n)$ and $A\in \R^{n\times n}$, it holds that
\begin{enumerate}
    \itemsep0em 
    \item For every $\gamma>0$, $\B_{\gamma f_1}=\frac{1}{\gamma}\B_{f_1}$
\item $\B_{f_1+f_2}=\B_{f_1}\sharp \,\B_{f_2}$,
\item $\B_{f_1\sharp f_2}=\B_{f_1}+ \B_{f_2}$,
\item $\B_{\max\{f_1, f_2\}}=\B_{f_1}\cap \B_{f_2}$,
\item $\B_{f_1\ssquare f_2}=\B_{f_1}\ssquare \B_{f_2}$
\item $\B_{f_1\circ A}=A^{-1}\B_{f_1}$
    \end{enumerate}
\end{lem}
From Lemma \ref{lemma:SetsProperties} and Lemma \ref{lemma:COrrespondenceBall} we trivially obtain the following statement.
\begin{lem}\label{lem:DualNormsAppendix}
Given $f_1,f_2\in \cV(\R^n)$ and $A\in \R^{n\times n}$ invertible, it holds that
\begin{enumerate}
    \itemsep0em 
    \item $\forall x\in \R^n,\,\,f_1(x)\leq f_2(x)\,\,\Leftrightarrow\,\,\forall x\in\R^n,\,\,f_2^\star(x)\leq f_1^\star(x),$
    \item For every $\gamma>0$, $(\gamma f_1)^\star=\frac{1}{\gamma}f_1^\star$
\item $(f_1+f_2)^\star=f_1^\star \;\sharp\; f_2^\star$,
\item $(f_1\sharp f_2)^\star=f_1^\star +f_2^\star$,
\item $(\max\{f_1,f_2\})^\star=f_1^\star\,\ssquare\, f_2^\star$,
\item $(f_1\ssquare f_2)^\star=\max\{f_1^\star,f_2^\star\}$,
\item $(f_1\circ A)^\star=f_1^\star\circ A^{-\top}$.
\end{enumerate}
\end{lem}
These results can be generalized for finite numbers of norms as summarized in the following table.
\renewcommand{\arraystretch}{1.3}
\begin{center}
\begin{tabular}{ |c|c|c| c|} 
 \hline
 Operation & Unit Ball & Dual & Dual Unit Ball \\ 
 \hline
$\gamma f$ & $\frac{1}{\gamma}\B_f$ & $\frac{1}{\gamma}f^\star$ & $\gamma \B^\circ$
  \\ 
 \hline
 $\oplus_i f_i$ & $\sharp_i \B_{f_i}$ & $\sharp_i f_i^\star$ & $\oplus_i \B^\circ_{f_i}$ \\  
 \hline
$\sharp_i f_i$ & $\oplus_i \B_{f_i}$ &  $\oplus_i f^\star_i$ & $\sharp_i \B^\circ_{f_i}$\\
 \hline
 $\max_i\{f_1\}$ & $\bigcap_i \B_{f_i}$ & $\square_i f_i^\star$ & $\square_i \B_{f_i}^\circ$\\
 \hline
 $\square_i f_i$ & $\square_i \B_{f_i}$ & $\max_i\{f_i^\star\}$ & $\bigcap \B^\circ_{f_i}$\\
 \hline
 $f\circ A$ & $A^{-1}\B_f$ & $f^\star\circ A^{-\top}$ & $A^{\top}\B^\circ$\\
 \hline
\end{tabular}
\end{center}
Finally, for the path-complete formalism, the following result is particularly helpful.
\begin{lem}[Dual Inequality]\label{lemma:DualInequallity}
Consider $f_1,f_2\in \cV(\R^n)$ and $A\in \R^{n\times n}$, then
\[
f_2(Ax)\leq f_1 (x)\,\,\,\forall\, x\in \R^n\,\,\Leftrightarrow\,\,f_1^\star(A^\top x)\leq f_2^\star(x),\,\,\,\forall \,x\in \R^n.
\]
\end{lem}
\begin{proof}
Recalling the definitions we have
$
f_2(Ax)\leq f_1(x),\,\,\forall \,x\in \R^n,
$
if and only if
\begin{equation}\label{eq:LemmaImpl1}
x\in \B_{f_1}\,\,\Rightarrow Ax\in \B_{f_2}.
\end{equation}
We show that~\eqref{eq:LemmaImpl1} is true if and only if
$
x\in \B^\circ_{f_2} \,\,\Rightarrow A^\top x\in \B^\circ_{f_1}$,
which is equivalent to $f_1^\star(A^\top x)\leq f_2^\star(x)$, $\forall \,x\in \R^n$.
Suppose \eqref{eq:LemmaImpl1} is true. Consider $x\in \B^\circ_{f_2}$, applying the definitions we have
\[
x\in \B^\circ_{f_2}\,\Leftrightarrow\, \inp{x}{y}\leq 1,\,\forall\,y\in \B_{f_2}\,\,\Rightarrow \,\,\inp{x}{Az}\leq 1\,\forall z\in \B_{f_1}\,\,\Leftrightarrow\,\,\inp{A^\top x}{z}\leq 1\,\forall z\in \B_{f_1}\,\,\Leftrightarrow\,\,A^\top x\in \B^\circ_{f_1}.
\]
The other direction is completely equivalent, once recalled that $(\B_{f_i}^\circ)^\circ=\B_{f_i}$.
\end{proof}

 \bibliographystyle{plain}
\bibliography{ArxivNAHS.bib}

\end{document}